\documentclass[preprint,12pt]{imsart-wp}
\startlocaldefs
\usepackage{amsthm,amsmath}
\usepackage{dsfont}
\usepackage[pdftex]{graphicx}
\usepackage[authoryear,round]{natbib}
\usepackage{mathrsfs}
\usepackage{amsfonts}
\usepackage{amssymb}
\usepackage{nicefrac}
\usepackage{tikz}
\usepackage{multirow}
\usetikzlibrary{graphs}
\newtheorem{prop}{Proposition}
\newtheorem{defn}{Definition}
\newcommand{\bea}{\begin{eqnarray}}
\newcommand{\eea}{\end{eqnarray}}
\newcommand{\be}{\begin{eqnarray*}}
\newcommand{\ee}{\end{eqnarray*}}

\theoremstyle{definition}

\endlocaldefs

\begin{document}

\begin{frontmatter}

\title{An irregularly spaced first-order moving average model}


\runtitle{Irregularly spaced MA(1) model}

 \author{C. Ojeda\textsuperscript{1}, \fnms{W.} \snm{Palma}\textsuperscript{2}, S. Eyheramendy\textsuperscript{2,3}, F. Elorrieta\textsuperscript{2,4}}
\vspace{0.3cm}
\address{\textsuperscript{1} Department of Statistics, Universidad del Valle, Cali, Colombia}
\address{\textsuperscript{2} Millennium Institute of Astrophysics,  Santiago, Chile}
\address{\textsuperscript{3} Faculty of Engineering and Sciences, Universidad Adolfo Iba\~nez, Santiago,  Chile}
\address{\textsuperscript{4} Department of Mathematics and Computer Science, Universidad de Santiago, Chile}

\vspace{0.1cm}

\runauthor{Ojeda, Palma, Eyheramendy, Elorrieta}

\begin{abstract}
A novel first-order moving-average model for analyzing time series observed at irregularly spaced intervals is introduced. Two definitions are presented, which are equivalent under Gaussianity. The first one relies on normally  distributed data and the specification of second-order moments. The second definition provided is more flexible in the sense that it allows for considering other distributional assumptions. The statistical properties are investigated along with the one-step linear predictors and their mean squared errors. It is established that the process is strictly stationary under  normality and weakly stationary in the general case.  Maximum likelihood and bootstrap estimation procedures are discussed and the finite-sample behavior of these estimates is assessed through Monte Carlo experiments. In these simulations, both methods perform well in terms of estimation bias and standard errors, even with relatively small sample sizes.  Moreover, we show that for non-Gaussian data, for t-Student and Generalized errors distributions, the parameters of the model can be estimated precisely by maximum likelihood. The proposed IMA model is compared to the continuous autoregressive moving average (CARMA) models, exhibiting good performance. Finally, the practical application and usefulness of the proposed model are illustrated with two real-life data examples.
\end{abstract}


\end{frontmatter}

\section{Introduction\label{sec:Introduction}}
Time series data observed at irregularly spaced intervals arise often in fields as diverse as astronomy, climatology, economics, finance, medical sciences, geophysics;  see for example \citet{moore1987experiences},\citet{Munoz1992}, \citet{Belcher1994}, \citet{kim2008fitting},\citet{Hand}, \citet{CORDUAS20081860}, \citet{HARVILL2013113}, \citet{Podgorski}, \citet{Babu}, \citet{eyheramendy2018irregular},  \citet{miller2019testing}, \citet{Edelmann}, \citet{caceres2019autoregressive}, \citet{refId0}, and \citet{ZHANG2020107019}, among others. In particular, in the climatology context, \citet{Mudelsee2014} points out that conventional time series analysis has largely ignored irregularly spaced structures that climate time series must  take into account. The statistical analysis of these data poses several difficulties. First, the overwhelming majority of the available time series methods assume regularly observed data, \citep[see, e.g.,][]{Brockwell1991,Hamilton1994,Box2016}. Second, when this assumption is dropped,  several technical problems arise including the issue of  formulating appropriate methodologies for carrying out  statistical inferences. Third, most of the currently available numerical algorithms for computing estimators and forecasts are based on the regularity of the data collection process. 

Irregularly spaced time series may occur in many ways. For example, data can be regularly spaced with missing observations. On the other hand, data can be truly irregularly spaced with no underlying regular sampling interval. Moreover,  a time series may exhibit both irregularly spaced measurements and missing values. This may occur, for example in astronomy where the measurement schedule may be irregularly spaced and the scheduled observation cannot be made at the specified date due to cloudiness.

Techniques considering time series in the presence of missing data are useful when data are regularly spaced with missing observations \citep[see, e.g.,][]{Parzen1963,Dunsmuir1983,Reinsel1987}. Nevertheless, these techniques cannot be applied if data are truly irregularly spaced. This case usually has been treated through two approaches. First, the irregularly spaced time series is transformed into regularly spaced  through interpolation and consequently conventional techniques are applied. \citet{Adorf1995} provides a summary of such transformations, which are frequently used to analyze astronomical data. However, interpolation methods typically produce bias (for instance, over-smoothing), changing the dynamic of the process. Second, irregularly spaced time series can be treated as discrete realizations of a continuous stochastic process \citep[see, e.g.,][]{Robinson1977,Parzen1984,Thornton2013}, but continuous time series models may present embedding and aliasing problems, see for example \citet{brockwell1999class}, \citet{huzii2007embedding} and \cite{Tomasson2015}. Consequently,  this paper proposes a novel approach  that allows for the  treatment of  first order moving averages (MA) structures with irregularly spaced times, without assuming a continuous-time or a regularly sampled underlying process. 

For illustration purposes, the proposed IMA model is compared to the continuous autoregressive moving average (CARMA) models. Observe that direct comparison is not possible due to that for a CARMA($p$,$q$) model with autoregressive order $p$ and moving average order $q$, it is required that $p>q$. Consequently, the IMA process is not a particular case of the CARMA class of models \citep[see,][]{Phadke1974,Thornton2013}. Thus, the IMA model is compared to  CARMA$(2,1)$ processes. Note that in the latter case, the model will have two additional parameters as compared to the IMA model. In these model comparisons, the IMA exhibits a better performance, assessed in terms of log-likelihood, AIC or Mean Squared Error  (MSE) criteria.

The remainder of the paper is organized as follows. Section~\ref{sec:IMA} defines a  class of stochastic processes called irregularly spaced first-order MA model for the treatment and analysis of unevenly spaced time series. Two definitions are presented, which under Gaussianity are equivalent. The first one relies on normally distributed data and the specification of second-order moments. The second one is more flexible in the sense that it allows to consider other distributional assumptions. This section also provides a state-space representation of the model along with one-step linear predictors and their mean squared errors. The maximum likelihood and bootstrap estimation methods are introduced in Section~\ref{sec:IMA-MLE}. The finite-sample behaviors of the proposed estimators are studied via Monte Carlo in Section~\ref{sec:IMA-MonteCarlo}, for Gaussian and non-Gaussian errors distributions. Two real-life data applications are discussed in Section \ref{sec:IMA-Application} while  conclusions are presented in Section~\ref{conclusions}. Finally, some technical results are provided in the Appendix section. 

\section{Definitions\label{sec:IMA}}

Let $\mathbb{T}=\{t_{1},t_{2},t_{3},\ldots\}$ be a set of observation times such that its consecutive differences $\Delta_{n+1}=t_{n+1}-t_{n}$, for $n\geq1$, are uniformly bounded and bounded away from zero. Thus, there are $\Delta_{L}>0$ and $\Delta_{U}<\infty$ such that $\Delta_{L}\leq\Delta_{n+1}\leq\Delta_{U}$ for all $n$. Without loss of generality, it is assumed that  $\Delta_{L}=1$. Otherwise, each $t_{n}$ can be re-scaled by $\Delta_{L}$. These conditions are compatibles with any physical measurement and determine $\mathbb{T}$ as a discrete, and therefore countable, subset of $\mathbb{R}$. Let  $\{x\}=\{X_{\tau},\tau\in\mathbb{T}\}$ be a discrete irregularly-spaced stochastic process and assume that the pattern of irregular spacing is independent to the stochastic properties of the process. When $\mathbb{T}$ is an arithmetic progression with $\Delta_{n+1}=\Delta$ for $n\geq1$, the process $\{x\}$ is called a discrete regularly-spaced stochastic process and no relevant information is lost by consider $\mathbb{T}=\mathbb{N}^{+}$. Henceforth, it is only refer to discrete stochastic processes or discrete time series. An irregularly (or unequally or unevenly) spaced time series is a finite realization of an irregularly spaced stochastic process. 

\subsection{An irregularly spaced first-order MA process of general form}

According to \citet[Proposition 2.4.2]{Azencott1986}, if  $m:\mathbb{T}\rightarrow\mathbb{R}$ is an arbitrary function and  $\Gamma:\mathbb{T}\times\mathbb{T}\rightarrow\mathbb{R}$ satisfying the following two conditions: 1)  $\Gamma(\tau,\iota)=\Gamma(\iota,\tau)$ for any $\tau,\iota\in\mathbb{T}$, and 2) the matrix $[\Gamma(\tau_{i},\tau_{j})]$ is non-negative definite for any finite set  $\{\tau_{1},\tau_{2},\ldots,\tau_{n}\}$ in  $\mathbb{T}$, then there is a real-valued Gaussian process $\{x\}$ for which its mean is $m$ and its variance-covariance matrix is $\Gamma$. With this result in mind,  the following stochastic process is considered. Let $m:\mathbb{T}\rightarrow\mathbb{R}$ be a function such that $m(t_{n})=0$, for any $t_{n}\in\mathbb{T}$. Further, let $\Gamma:\mathbb{T}\times\mathbb{T}\rightarrow\mathbb{R}$ be a function such that, for any $t_{n},t_{s}\in\mathbb{T}$,
\[
\Gamma(t_{n},t_{s})=\begin{cases}
\gamma_{0}, & \left|n-s\right|=0,\\
\gamma_{1,\Delta_{\max\{n,s\}}}, & \left|n-s\right|=1,\\
0, & \left|n-s\right|\geq2.
\end{cases}
\]
Note that $\Gamma$ can be represented by an infinite symmetric real-valued tridiagonal matrix as
\begin{equation}
\boldsymbol{\Gamma}=\begin{bmatrix}\gamma_{0} & \gamma_{1,\Delta_{2}} & 0 & 0 & \cdots\\
\gamma_{1,\Delta_{2}} & \gamma_{0} & \gamma_{1,\Delta_{3}} & 0\\
0 & \gamma_{1,\Delta_{3}} & \gamma_{0} & \gamma_{1,\Delta_{4}}\\
0 & 0 & \gamma_{1,\Delta_{4}} & \gamma_{0}\\
\vdots &  &  &  & \ddots
\end{bmatrix}.\label{eq:InfGenGamma}
\end{equation}

Now, let $\boldsymbol{\Gamma}_{n}$ be the $n\times n$ truncation of $\boldsymbol{\Gamma}$. From \citet[Theorem 3.2]{Ismail1991}, $\boldsymbol{\Gamma}_{n}$ is positive definite (assuming $\gamma_{1,\Delta_{j}}\neq0$) if $\gamma_{0}>0$ and $(\nicefrac{\gamma_{1,\Delta_{j}}}{\gamma_{0}})^{2}\leq\nicefrac{1}{4}$ for $j=2,\ldots,n$. Thus, for $\boldsymbol{\Gamma}$, if
\begin{equation}
\gamma_{0}>0\quad\textrm{and}\quad\left(\frac{\gamma_{1,\Delta_{n+1}}}{\gamma_{0}}\right)^{2}\leq\nicefrac{1}{4},\;\textrm{for}\;n\geq1,\;\textrm{with}\;\gamma_{1,\Delta_{n+1}}\neq0,\label{eq:ProcessConditions}
\end{equation}
then there is a real-valued Gaussian process $\{x\}$, with mean $0$ and variance-covariance matrix $\Gamma$. This process is called an irregularly spaced first-order MA process of general form. Specific expressions for $\gamma_{0}$ and $\gamma_{1,\Delta_{n+1}}$ satisfying \eqref{eq:ProcessConditions} are given in the next section. The goal is to define  a stationary irregularly spaced stochastic process for which  the conventional first-order MA process is obtained when $\Delta_{n+1}=1$ for all $n\geq1$.

\subsection{An irregularly spaced first-order MA model}

A class of {\em irregularly spaced first-order moving average} (IMA) processes is defined next. Observe that the definition is based on either a distributional or a constructionist viewpoint being this second more general.

In \eqref{eq:InfGenGamma}, $\gamma_{0}$ and $\gamma_{1,\Delta_{n+1}}$, for $n\geq1$, represent the variance and the first-order auto-covariances, respectively. In order to satisfy \eqref{eq:ProcessConditions},  the variance is defined as $\gamma_{0}=\sigma^{2}(1+\theta^{2})$ and the first-order auto-covariances as $\gamma_{1,\Delta_{n+1}}=\sigma^{2}\theta^{\Delta_{n+1}}$, where $\sigma^{2}>0$ and $0\leq\theta<1$. Hence,  a particular stationary irregularly spaced stochastic process with variance-covariance matrix given by
\begin{equation}
\boldsymbol{\Gamma}=\sigma^{2}\begin{bmatrix}1+\theta^{2} & \theta^{\Delta_{2}} & 0 & 0 & \cdots\\
\theta^{\Delta_{2}} & 1+\theta^{2} & \theta^{\Delta_{3}} & 0\\
0 & \theta^{\Delta_{3}} & 1+\theta^{2} & \theta^{\Delta_{4}}\\
0 & 0 & \theta^{\Delta_{4}} & 1+\theta^{2}\\
\vdots &  &  &  & \ddots
\end{bmatrix}\label{eq:InfIMAGamma}
\end{equation}
is obtained, which contains the conventional first-order MA process as a special case when $\Delta_{n+1}=1$ for all $n\geq1$. This Gaussian process is an irregularly spaced first-order MA process.

\begin{defn}[IMA--distributional viewpoint]
\label{def:-IMA-Distributional} Let $\mathbb{T}=\{t_{1},t_{2},t_{3},\ldots\}$ be a set such that its consecutive differences $\Delta_{n+1}=t_{n+1}-t_{n}$, for $n\geq1$, are uniformly bounded and bounded away from zero. The IMA process $\{X_{t_{n}},t_{n}\in\mathbb{T}\}$ is defined as a Gaussian process with mean $\boldsymbol{0}$ and covariance \eqref{eq:InfIMAGamma} with $\sigma^{2}>0$ and $0\leq\theta<1$. It is said that $\{X_{t_{n}},t_{n}\in\mathbb{T}\}$ is an IMA process with mean $\mu$ if $\{X_{t_{n}}-\mu,t_{n}\in\mathbb{T}\}$ is an IMA process.
\end{defn}

Since the definition above assumes Gaussian distribution, which only relies on the specification of moments up to second-order,  an alternative and more flexible definition of an IMA process is considered next. This approach is known as a constructionist viewpoint of the process \citep{Spanos1999}. Appendix \ref{sec:Constructionist-viewpoint} contains details about how  this model is built.

\begin{defn}[IMA--constructionist viewpoint]
\label{def:IMA-Constructionist}Let $\{\varepsilon_{t_{n}}\}_{n\geq1}$ be a sequence of uncorrelated random variables with mean 0 and variance $\sigma^{2}c_{n}(\theta)$. Here, $\sigma^{2}>0$, $0\leq\theta<1$, $c_{1}(\theta)=1+\theta^{2}$ and
\[
c_{n}(\theta)=1+\theta^{2}-\frac{\theta^{2\Delta_{n}}}{c_{n-1}(\theta)}\quad\text{for}\quad n\geq2,
\]
where $\Delta_{n}=t_{n}-t_{n-1}$. The process $\{X_{t_{n}},t_{n}\in\mathbb{T}\}$, with $\mathbb{T}$ as defined in Definition \ref{def:-IMA-Distributional}, is said to be an IMA process if $X_{t_{1}}=\varepsilon_{t_{1}}$ and, for $n\geq2$,
\begin{equation}
X_{t_{n}}=\varepsilon_{t_{n}}+\frac{\theta^{\Delta_{n}}}{c_{n-1}(\theta)}\varepsilon_{t_{n-1}}.\label{eq:IMA}
\end{equation}
It is said  that $\{X_{t_{n}},t_{n}\in\mathbb{T}\}$ is an IMA process with mean $\mu$ if $\{X_{t_{n}}-\mu,t_{n}\in\mathbb{T}\}$ is an IMA process.
\end{defn}

Observe that under Gaussianity, the definition of the IMA model provided by the constructionist viewpoint is equivalent to the definition given by distributional viewpoint  since in this case \eqref{eq:IMA} is a Gaussian process with variance $\gamma_{0}=\sigma^{2}(1+\theta^{2})$ and first-order auto-covariances $\gamma_{1,\Delta_{n+1}}=\sigma^{2}\theta^{\Delta_{n+1}}$, for $n\geq1$ (see Appendix \ref{sec:Constructionist-viewpoint} for more details). Additionally, it is interesting to note that the sequence $c_{n}(\theta)$ is known as a general backward continued fraction \citep{Kilic2008}, which with $\Delta_{n}\geq1$, for all $n$, satisfies (proof in Appendix \ref{sec:On-gbcf-IMA})
\[
1\leq c_{n}(\theta)<2\quad\text{for}\;n\geq1.
\]

\subsection{Properties}

If $\boldsymbol{X}_{n}=[X_{t_{1}},\ldots,X_{t_{n}}]^{\prime}$ is a random vector from an IMA process, then $\boldsymbol{X}_{n}$ is a Gaussian random vector with mean $\boldsymbol{\textrm{m}}_{n}=\boldsymbol{0}$ and tridiagonal covariance matrix
\begin{equation}
\boldsymbol{\Gamma}_{n}=\sigma^{2}\begin{bmatrix}1+\theta^{2} & \theta^{\Delta_{2}} & \cdots & 0 & 0\\
\theta^{\Delta_{2}} & 1+\theta^{2} & \cdots & 0 & 0\\
\vdots & \vdots & \ddots & \vdots & \vdots\\
0 & 0 & \cdots & 1+\theta^{2} & \theta^{\Delta_{n}}\\
0 & 0 & \cdots & \theta^{\Delta_{n}} & 1+\theta^{2}
\end{bmatrix}.\label{eq:Gamman}
\end{equation}
Thus, the IMA process is a weakly stationary Gaussian process and therefore strictly stationary. On the other hand, under Definition 2, the process has constant mean and its covariance structure depends only on the time differences $\Delta_n$. Thus, the process is weakly stationary.


\subsection{State-space representation\label{sec:IMA-StateSpace}}

Using the same notation given in Definition \ref{def:IMA-Constructionist},  a state-space representation of the model \eqref{eq:IMA} can be obtained. This representation has the minimal dimension of the state vector and
is given by
\begin{align*}
\alpha_{t_{n+1}} & =\frac{\theta^{\Delta_{n+1}}}{c_{n}(\theta)}\varepsilon_{t_{n}},\\
X_{t_{n}} & =\alpha_{t_{n}}+\varepsilon_{t_{n}}
\end{align*}
for $n\geq1$ with $\alpha_{t_{1}}=0$. Note that, in this representation, the transition and measurement equation disturbances are correlated. As suggested by \citet{Harvey1989}, in order to get a new system on which these disturbances are uncorrelated, the following expressions are used,
\begin{align}\label{eq:IMA-TranEqUncorrelated}
\alpha_{t_{n+1}} & =-\frac{\theta^{\Delta_{n+1}}}{c_{n}(\theta)}\alpha_{t_{n}}+\frac{\theta^{\Delta_{n+1}}}{c_{n}(\theta)}X_{t_{n}},\\
X_{t_{n}} & =\alpha_{t_{n}}+\varepsilon_{t_{n}}.\notag
\end{align}
The inclusion of $X_{t_{n}}$ in \eqref{eq:IMA-TranEqUncorrelated} does not affect the Kalman filter, as $X_{t_{n}}$ is known at time $t_{n}$.

\subsection{Prediction\label{sec:IMA-Prediction}}

The one-step linear predictors are defined as $\hat{X}_{t_{1}}=0$
and
\[
\hat{X}_{t_{n+1}}=\phi_{n1}X_{t_{1}}+\cdots+\phi_{nn}X_{t_{n}},\;n\geq1,
\]
where $\phi_{n1},\ldots,\phi_{nn}$ satisfy the prediction equations
\begin{equation}
\boldsymbol{\Gamma}_{n}\boldsymbol{\phi}_{n}=\boldsymbol{\gamma}_{n}.\label{eq:OneStepPredictionEquations}
\end{equation}

In terms of the IMA process the matrix  $\boldsymbol{\Gamma}_{n}$ is given by  \eqref{eq:Gamman} and,
{\footnotesize{}
\[
\boldsymbol{\phi}_{n}=\begin{bmatrix}\phi_{n1}\\
\phi_{n2}\\
\vdots\\
\phi_{nn}
\end{bmatrix}\,\textrm{and}\,\boldsymbol{\gamma}_{n}=\sigma^{2}\begin{bmatrix}0\\
0\\
\vdots\\
\theta^{\Delta_{n+1}}
\end{bmatrix}.
\]
}

The {\em mean squared errors} (MSE) are $\nu_{n+1}=\textrm{E}[(X_{t_{n+1}}-\hat{X}_{t_{n+1}})^{2}]=\gamma_{0}-\boldsymbol{\gamma}_{n}^{\prime}\boldsymbol{\Gamma}_{n}^{-1}\boldsymbol{\gamma}_{n}$ with $\nu_{1}=\gamma_{0}$. From \citet[Proposition 5.1.1]{Brockwell1991},
there is exactly one solution of \eqref{eq:OneStepPredictionEquations} which is given by
\[
\boldsymbol{\phi}_{n}=\boldsymbol{\Gamma}_{n}^{-1}\boldsymbol{\gamma}_{n}.
\]

A useful technique for solving the prediction equations is the innovations algorithm \citep{Brockwell1991}. This procedure  gives the coefficients of $X_{t_{n}}-\hat{X}_{t_{n}},\ldots,X_{t_{1}}-\hat{X}_{t_{1}}$, in the alternative expansion $\hat{X}_{t_{n+1}}=\sum_{j=1}^{n}\theta_{nj}(X_{t_{n+1-j}}-\hat{X}_{t_{n+1-j}})$. Hence, by using this algorithm (see Appendix \ref{sec:Innovations-algorithm}) the following expressions are obtained:
\begin{align*}
\theta_{n,j} & =0,\quad2\leq j\leq n,\\
\theta_{n,1} & =\frac{\gamma_{1,\Delta_{n+1}}}{\upsilon_{n}}\quad\textrm{and}\\
\upsilon_{n+1} & =\gamma_{0}-\theta_{n,1}^{2}\upsilon_{n}=\gamma_{0}-\frac{\gamma_{1,\Delta_{n+1}}^{2}}{\upsilon_{n}},\quad\textrm{with}\;\upsilon_{1}=\gamma_{0}.
\end{align*}

Specifically, for the IMA process,  $\hat{X}_{t_{1}}(\theta)=0$ with mean squared error $\sigma^{2}c_{1}(\theta)$ and
\[
\hat{X}_{t_{n+1}}(\theta)=\frac{\theta^{\Delta_{n+1}}}{c_{n}(\theta)}(X_{t_{n}}-\hat{X}_{t_{n}}(\theta)),\;n\geq1,
\]
with mean squared errors $\sigma^{2}c_{n+1}(\theta)$.

\section{Estimation\label{sec:IMA-MLE}}

Let $X_{t}$ be observed at points $t_{1},\ldots,t_{\textrm{N}}$. The log likelihood is then
\[
-\frac{1}{2}\textrm{N}\log2\pi\sigma^{2}-\frac{1}{2}\sum_{n=2}^{\textrm{N}}\log c_{n}(\theta)-\frac{1}{2}\sum_{n=2}^{\textrm{N}}\frac{(X_{t_{n}}-\hat{X}_{t_{n}}(\theta))^{2}}{\sigma^{2}c_{n}(\theta)},
\]
where $\theta$ and $\sigma^{2}$ are any admissible parameter values. Now,  differentiating with respect to $\sigma^{2}$ and replacing the expression with its estimator,  the reduced likelihood \citep{Brockwell1991} is obtained:
\begin{align*}
    q_{\textrm{N}}(\theta) & =\log\hat{\sigma}_{\textrm{N}}^{2}(\theta)+\frac{1}{\textrm{N}}\sum_{n=2}^{\textrm{N}}\log c_{n}(\theta),\\
    \hat{\sigma}_{\textrm{N}}^{2}(\theta) & =\frac{1}{\textrm{N}}\sum_{n=2}^{\textrm{N}}\frac{(X_{t_{n}}-\hat{X}_{t_{n}}(\theta))^{2}}{c_{n}\left(\theta\right)}.
\end{align*}
The Maximum Likelihood (ML) estimate of $\theta$, $\hat{\theta}_{\textrm{N}}$, is the value minimizing $q_{\textrm{N}}(\theta)$, and the estimate of $\sigma^{2}$ is $\hat{\sigma}_{\textrm{N}}^{2}=\hat{\sigma}_{\textrm{N}}^{2}(\hat{\theta}_{\textrm{N}})$. The optimization can be done through the method proposed by \citet{Byrd1995}, which allows general box constraints. Specifically, $q_{\textrm{N}}(\theta)$ can be minimized under the constraint $0\leq\theta<1$. Also, this method allows for finding the numerically differentiated Hessian matrix at the solution given. Solving it,  the estimated standard errors can be obtained.

\subsection{Bootstrap\label{sec:IMA-Bootstrap}}

In order to carry out statistical inference it is necessary  to derive the distributions of the statistics used for the estimation of the parameters from the data, but these calculations are often cumbersome. Additionally,  if $N$ is small, or if the parameters are close to the boundaries, the asymptotic approximations can be quite poor \citep{Shumway2017}. Also, in the irregularly spaced time case, the asymptotic approximations need to establish strong conditions such as in \citet[Section 4, pp.12]{Robinson1977} which are difficult to meet. To overcome these difficulties and to get approximations of the finite sample distributions,  the bootstrap method can be used. The usual procedure for applying the bootstrap technique in the context of  time series is fitting a suitable model to the data, to obtain the residuals from the fitted model, and then generate a new series by incorporating random samples from the residuals into the fitted model. The residuals are typically centered to have the same mean as the innovations of the model.

Following  \citet{Bose1990}, for $n\geq2$, define the estimated innovations or one-step prediction residuals, as
\[
e_{t_{n}}=X_{t_{n}}+\sum_{j=1}^{n-1}(-1)^{j}\frac{\prod_{k=n-j+1}^{n}\theta^{\Delta_{k}}}{\prod_{l=n-j}^{n-1}c_{l}(\theta)}X_{t_{n-j}}
\]
and $e_{t_{1}}=X_{t_{1}}$. Using the structure of the process and assuming that the fitted model is, in fact, the true model for the data, the residuals are given by,
\[
e_{t_{n}}=\varepsilon_{t_{n}}+(-1)^{n-2}\frac{\prod_{k=2}^{n}\theta^{\Delta_{k}}}{\prod_{l=1}^{n-1}c_{l}(\theta)}\varepsilon_{t_{1}}.
\]
Hence $e_{t_{n}}$ and $\varepsilon_{t_{n}}$ are close for all large $n$ if $0\leq\theta<1$, which shows that resampling is proper in this situation.

Next,  the bootstrap method is applied to the estimation of $\theta$ in the IMA process with $\sigma^{2}=1$. Let $X_{t}$ be observed at points $t_{1},\ldots,t_{\textrm{N}}$ and consider $\hat{\theta}_{\textrm{N}}$ as the ML estimator. The standardized estimated innovations are
\[
e_{t_{n}}^{s}=\frac{X_{t_{n}}-\hat{X}_{t_{n}}(\hat{\theta}_{\textrm{N}})}{\sqrt{c_{n}(\hat{\theta}_{\textrm{N}})}},
\]
for $n=2,\ldots,\textrm{N}$. The so-called model-based resampling might proceed by equi-probable sampling with replacement from centered residuals $e_{t_{2}}^{sc}=e_{t_{2}}^{s}-\bar{e},\ldots,e_{t_{\textrm{N}}}^{sc}=e_{t_{\textrm{N}}}^{s}-\bar{e}$, where $\bar{e}=\sum_{n=2}^{N}\nicefrac{e_{t_{n}}^{s}}{\textrm{N}-1}$, to obtain simulated innovations $e_{t_{1}^{\ast}}^{sc},\ldots,e_{t_{\textrm{N}}^{\ast}}^{sc}$, with $t_{1}^{\ast},\ldots,t_{\textrm{N}}^{\ast}$  the sampled times, and then setting
\begin{gather*}
X_{t_{1}}^{\ast}=\sqrt{c_{1}(\hat{\theta}_{\textrm{N}})}e_{t_{1}^{\ast}}^{sc},\\
X_{t_{n}}^{\ast}=\sqrt{c_{n}(\hat{\theta}_{\textrm{N}})}e_{t_{n}^{\ast}}^{sc}+\frac{\hat{\theta}_{\textrm{N}}^{\Delta_{n}}}{c_{n-1}(\hat{\theta}_{\textrm{N}})}\sqrt{c_{n-1}(\hat{\theta}_{\textrm{N}})}e_{t_{n-1}^{\ast}}^{sc},\quad\textrm{for}\;n=2,\ldots,\textrm{N}.
\end{gather*}

Next,  the parameters are estimated via  ML assuming that the data are $X_{t_{n}}^{\ast}$. This process is repeated  $\textrm{B}$ times, generating a collection of bootstrapped parameter estimates. Finally,  the finite sample distribution of the estimator, $\hat{\theta}_{\textrm{N}}$, is estimated from the bootstrapped parameter values.

\section{Simulations\label{sec:IMA-MonteCarlo}}

This section provides a Monte Carlo study that assesses the finite-sample performance of the maximum likelihood and bootstrap estimators. The simulation setting  consider $\sigma^{2}=1$, $\theta\in\left\{ 0.1,0.5,0.9\right\} $ and $\textrm{N}\in\left\{ 100,500,1500\right\} $, where $\textrm{N}$ represent the length of the series. Furthermore,  $\textrm{M}=1000$ trajectories, $\{\textrm{S}_{m}\}_{m=1}^{\textrm{M}}$ are simulated, and  $\theta$ is estimated. For each setup,  regular as well as irregular spaced times $t_{1},\ldots,t_{\textrm{N}}$ are considered, where $t_{n}-t_{n-1}\overset{\textrm{ind}}{\sim}1+\textrm{exp}(\lambda=1)$, for $n=2,\ldots,\textrm{N}$. In Figure \ref{TrajectoryIMAExample}, we present an IMA trajectory example with the tick marks of irregularly
spaced times.

\begin{figure}
\begin{centering}
\includegraphics[height=3in,width=5.5in]{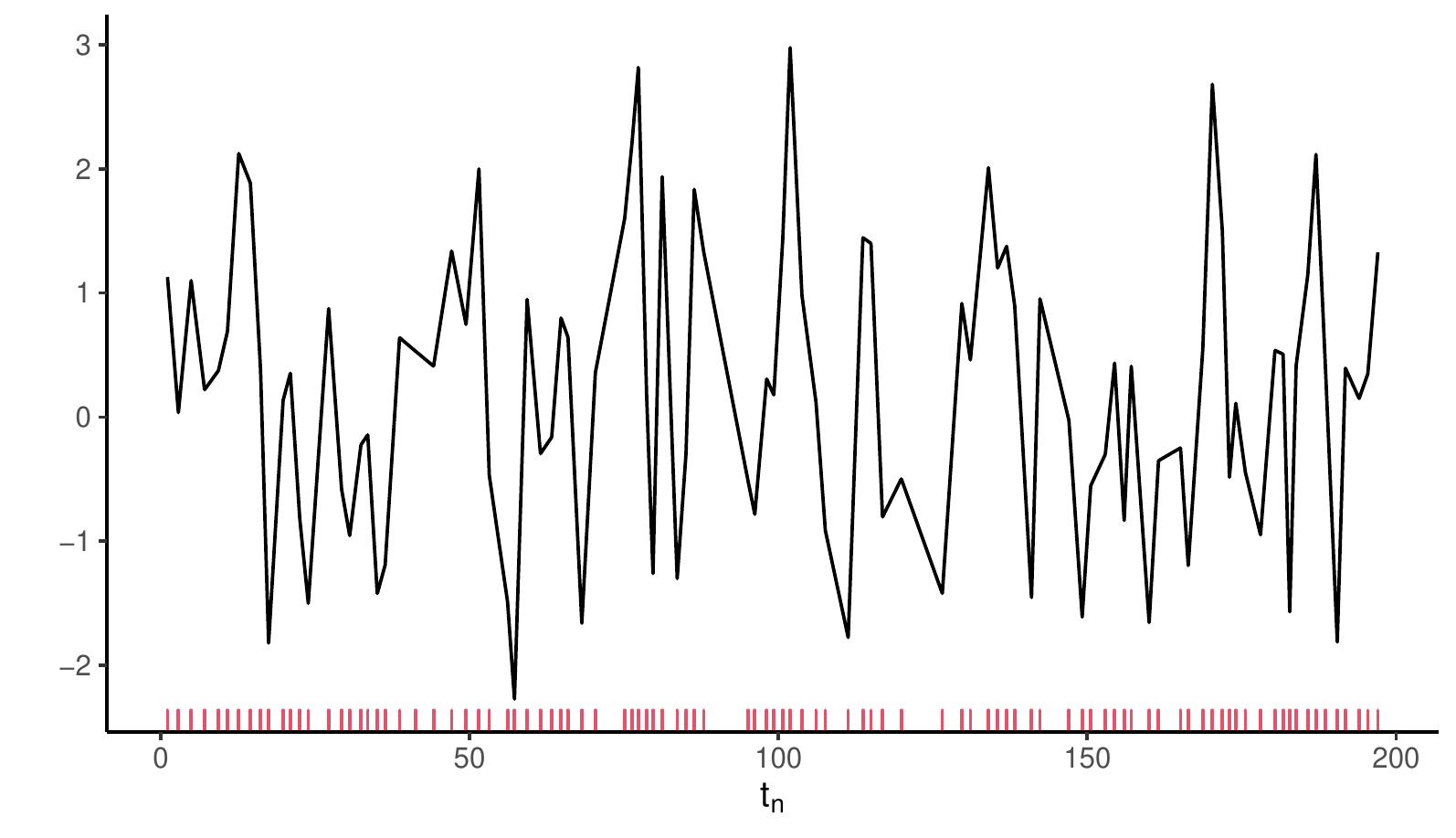}
\end{centering}
\caption{IMA trajectory example with $\theta=0.5$, $\sigma^{2}=1$ and $\textrm{N}=100$. On the bottom, we found the tick marks of irregularly spaced times.}
\label{TrajectoryIMAExample}
\end{figure}

\subsection{ML estimators}

Let $\hat{\theta}_{m}^{\textrm{MLE}}$ be the ML estimation and $\widehat{\textrm{se}}(\hat{\theta}_{m}^{\textrm{MLE}})$ be the estimated standard error for the $m$-th trajectory. The standard error is estimated by curvature of the likelihood surface at $\hat{\theta}_{m}^{\textrm{MLE}}$. The mean value of the $\textrm{M}$ maximum likelihood estimations is computed as,
\[
\hat{\theta}^{\textrm{MLE}}=\frac{1}{\textrm{M}}\sum_{m=1}^{\textrm{M}}\hat{\theta}_{m}^{\textrm{MLE}}\quad\textrm{and}\quad\widehat{\textrm{se}}(\hat{\theta}^{\textrm{MLE}})=\frac{1}{\textrm{M}}\sum_{m=1}^{\textrm{M}}\widehat{\textrm{se}}(\hat{\theta}_{m}^{\textrm{MLE}}).
\]

\subsection{Bootstrap estimators}

For each trajectory, we simulated $\textrm{B}=500$
bootstrap trajectories represented by $\{\{\textrm{S}_{m,b}\}_{b=1}^{\textrm{B}}\}_{m=1}^{\textrm{M}}$. Then, the sequence $\{\{\hat{\theta}_{m,b}^{\textrm{b}}\}_{b=1}^{\textrm{B}}\}_{m=1}^{\textrm{M}}$ (the ML estimations) is obtained. The bootstrap estimation and their estimated variance error are defined as
\[
\hat{\theta}_{m}^{\textrm{b}}=\frac{1}{\textrm{B}}\sum_{b=1}^{\textrm{B}}\hat{\theta}_{m,b}^{\textrm{b}}\quad\textrm{and}\quad\widehat{\textrm{se}}^{2}(\hat{\theta}_{m}^{\textrm{b}})=\frac{1}{\textrm{B}-1}\sum_{b=1}^{\textrm{B}}(\hat{\theta}_{m,b}^{\textrm{b}}-\hat{\theta}_{m}^{\textrm{b}})^{2},
\]
for $m=1,\ldots,\textrm{M}$. Finally, the mean value of the $\textrm{M}$ bootstrap estimations is given by
\[
\hat{\theta}^{\textrm{b}}=\frac{1}{\textrm{M}}\sum_{m=1}^{\textrm{M}}\hat{\theta}_{m}^{\textrm{b}}\quad\textrm{and}\quad\widehat{\textrm{se}}(\hat{\theta}^{\textrm{b}})=\frac{1}{\textrm{M}}\sum_{m=1}^{\textrm{M}}\widehat{\textrm{se}}(\hat{\theta}_{m}^{\textrm{b}}).
\]

\subsection{Performance measures}

Besides, as a measure of the estimator performance, the Root Mean Squared Error (RMSE), and the Coefficient of Variation (CV) are considered. Recall that in this case the RMSE and the CV may be written as
\begin{gather*}
\textrm{RMSE}_{\hat{\theta}^{\textrm{MLE}}}=(\widehat{\textrm{se}}(\hat{\theta}^{\textrm{MLE}})^{2}+\textrm{bias}_{\hat{\theta}^{\textrm{MLE}}}^{2})^{\nicefrac{1}{2}},\;\textrm{and}\\
\textrm{CV}_{\hat{\theta}^{\textrm{MLE}}}=\frac{\widehat{\textrm{se}}(\hat{\theta}^{\textrm{MLE}})}{\left|\hat{\theta}^{\textrm{MLE}}\right|},
\end{gather*}
where $\textrm{bias}_{\hat{\theta}^{\textrm{MLE}}}=\hat{\theta}^{\textrm{MLE}}-\theta$. Finally,  an approximate variance estimator is given by
\[
\widetilde{\textrm{se}}^{2}(\hat{\theta}^{\textrm{MLE}})=\frac{1}{\textrm{M}-1}\sum_{m=1}^{\textrm{M}}(\hat{\theta}_{m}^{\textrm{MLE}}-\hat{\theta}^{\textrm{MLE}})^{2}.
\]

$\textrm{RMSE}_{\hat{\theta}^{\textrm{b}}}$, $\textrm{CV}_{\hat{\theta}^{\textrm{b}}}$, $\textrm{bias}_{\hat{\theta}^{\textrm{b}}}$ and $\widetilde{\textrm{se}}^{2}(\hat{\theta}^{\textrm{b}})$ for the bootstrap case are defined analogously. Figure \ref{GeneralSchemeMC} shows the workflow of the simulation study.

\begin{figure}
\begin{centering}
\begin{tikzpicture}[scale = 0.85]
\node{IMA}[sibling distance=4.5cm]
 child{node{$\textrm{S}_{1}$}
  child{node{$\{\hat{\theta}_{1}^{\textrm{MLE}},\widehat{\textrm{se}}(\hat{\theta}_{1}^{\textrm{MLE}})\}$}[sibling distance=1.5cm]
   child{node{$\textrm{S}_{1,1}$}
    child{node{$\{\hat{\theta}_{1,1}^{\textrm{b}}\}$}}}
   child{node{$\cdots\;\;\textrm{S}_{1,b}\;\;\cdots$}
    child{node{$\cdots\{\hat{\theta}_{1,b}^{\textrm{b}}\}\cdots$}[level distance = 0.7cm]
     child{node{$\overset{\underbrace{\hspace{3cm}}}{\{\hat{\theta}_{1}^{\textrm{b}},\widehat{\textrm{se}}(\hat{\theta}_{1}^{\textrm{b}})\}}$} edge from parent [draw=none]}}}
   child{node{$\textrm{S}_{1,\textrm{B}}$}
    child{node{$\{\hat{\theta}_{1,\textrm{B}}^{\textrm{b}}\}$}}}}}
 child{node{$\cdots\qquad\qquad\textrm{S}_{m}\qquad\qquad\cdots$}
  child{node{$\cdots\qquad\{\hat{\theta}_{m}^{\textrm{MLE}},\widehat{\textrm{se}}(\hat{\theta}_{m}^{\textrm{MLE}})\}\qquad\cdots$}[sibling   distance=1.5cm]
   child{node{$\textrm{S}_{m,1}$}
    child{node{$\{\hat{\theta}_{m,1}^{\textrm{b}}\}$}}}
   child{node{$\cdots\;\;\textrm{S}_{m,b}\;\;\cdots$}
    child{node{$\cdots\{\hat{\theta}_{m,b}^{\textrm{b}}\}\cdots$}[level distance = 0.7cm]
     child{node{$\overset{\underbrace{\hspace{3cm}}}{\{\hat{\theta}_{2}^{\textrm{b}},\widehat{\textrm{se}}(\hat{\theta}_{2}^{\textrm{b}})\}}$} edge from parent [draw=none]}}}
   child{node{$\textrm{S}_{m,\textrm{B}}$}
    child{node{$\{\hat{\theta}_{m,\textrm{B}}^{\textrm{b}}\}$}}}}}
 child{node{$\textrm{S}_{\textrm{M}}$}
  child{node{$\{\hat{\theta}_{\textrm{M}}^{\textrm{MLE}},\widehat{\textrm{se}}(\hat{\theta}_{\textrm{M}}^{\textrm{MLE}})\}$}[sibling distance=1.5cm]
   child{node{$\textrm{S}_{\textrm{M},1}$}
    child{node{$\{\hat{\theta}_{\textrm{M},1}^{\textrm{b}}\}$}}}
   child{node{$\cdots\;\;\textrm{S}_{\textrm{M},b}\;\;\cdots$}
    child{node{$\cdots\{\hat{\theta}_{\textrm{M},b}^{\textrm{b}}\}\cdots$}[level distance = 0.7cm]
     child{node{$\overset{\underbrace{\hspace{3cm}}}{\{\hat{\theta}_{\textrm{M}}^{\textrm{b}},\widehat{\textrm{se}}(\hat{\theta}_{\textrm{M}}^{\textrm{b}})\}}$} edge from parent [draw=none]}}}
   child{node{$\textrm{S}_{\textrm{M},\textrm{B}}$}
    child{node{$\{\hat{\theta}_{\textrm{M},\textrm{B}}^{\textrm{b}}\}$}}}}};
\end{tikzpicture}
\end{centering}
\caption{General scheme of Monte Carlo study. Here, we show how we got the pairs $\{\hat{\theta}_{m}^{\textrm{mle}},\hat{\textrm{se}}(\hat{\theta}_{m}^{\textrm{mle}})\}_{m=1}^{\textrm{M}}$ and $\{\hat{\theta}_{m}^{\textrm{b}},\hat{\textrm{se}}(\hat{\theta}_{m}^{\textrm{b}})\}_{m=1}^{\textrm{M}}$
for each trajectory.}
\label{GeneralSchemeMC}
\end{figure}
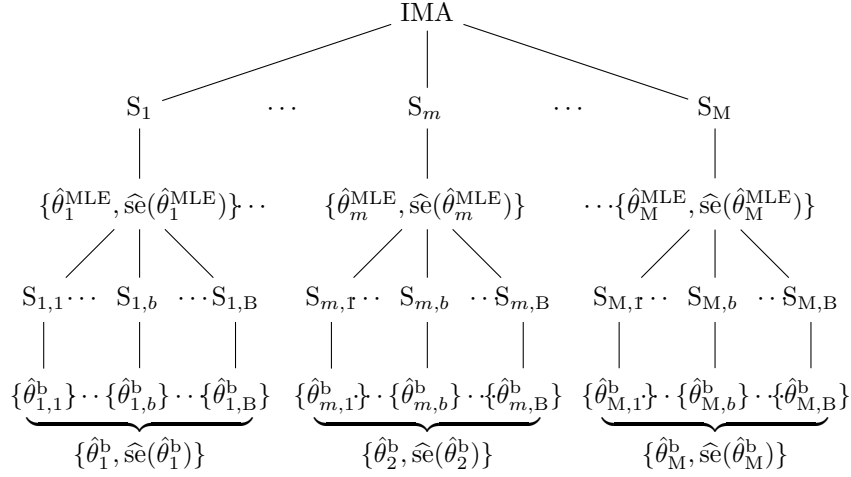

Next,  the results for the irregularly spaced time case are reported (see Appendix \ref{sec:IMA-MC-RegularySpacedTimes} for the regularly spaced case). First, the simulated finite sample distributions for maximum likelihood and bootstrap estimators are shown. Second,   the measures of performance for both estimators, are shown. Separately,  the Monte Carlo Error (MCE) is estimated  for every simulation via asymptotic theory \citep[see,][]{Koehler2009}, 
\[
\frac{\widetilde{\textrm{se}}(\hat{\theta}^{\textrm{MLE}})}{\sqrt{M}},
\]
and  the maximum value is reported in each table. The MCE is a estimation of the standard deviation of the Monte Carlo estimator, taken across repetitions of the simulation, where each simulation is based on the same design and consists of $M$ replications.

\subsection{Finite sample distributions for ML and bootstrap estimators}

In Figure \ref{IMA-simulatedFiniteSampleDistributions}, the simulated finite sample distributions for maximum likelihood and bootstrap estimators are reported. Observe that both estimators behave  well in these Monte Carlo experiments. They seem to be unbiased since the mean of the estimated values are close to their theoretical counterparts. However, for small values of $\theta_{0}$, the  estimates exhibit greater variability.

\begin{figure}
\begin{centering}
\includegraphics[height=3in,width=5.5in]{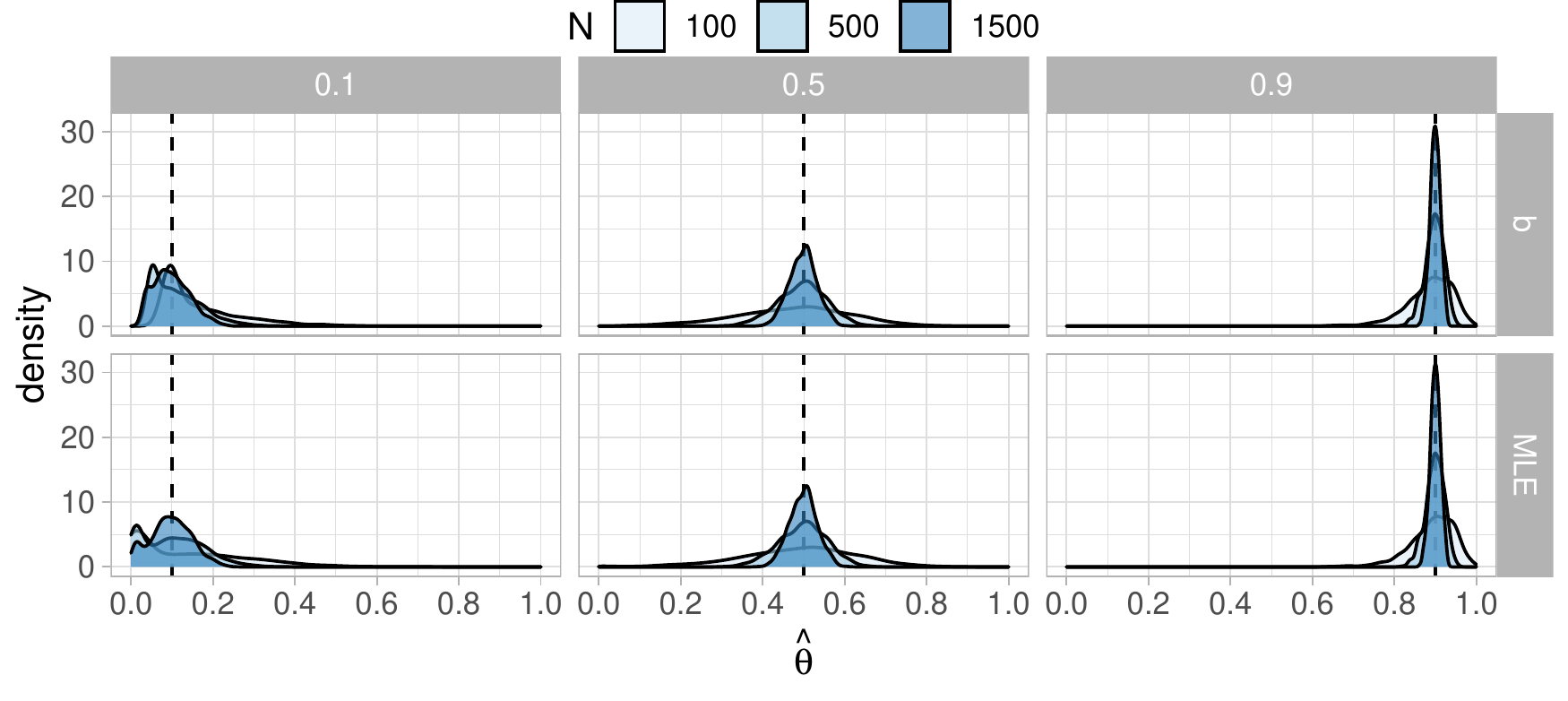}
\end{centering}
\caption{Simulated finite sample distributions. On the top, the MLE case. On the bottom, the bootstrap case.}
\label{IMA-simulatedFiniteSampleDistributions}
\end{figure}

Tables \ref{IMA-MC-Table-IrregularSpaced-MLE} and \ref{IMA-MC-Table-IrregularSpaced-b} exhibit the performance measures of the estimators, for maximum likelihood and bootstrap methods, respectively. Bias and variance are smaller when $\textrm{N}$ increases as expected. Both methods provide good estimations for the standard error. Also, the irregularly spaced times seem to increase the estimation variability. In Appendix \ref{sec:IMA-MC-RegularySpacedTimes}, it is presented the regularly spaced times case. In this case, the IMA model is reduced to the conventional MA model, and the results are equally consistent.

\begin{center}
	\begin{table}
		\caption{Monte Carlo results for the ML estimator with irregularly spaced times. The MCE estimated is $0.005$.}
		\label{IMA-MC-Table-IrregularSpaced-MLE}
		\begin{centering}
			\begin{tabular}{|c|c|c|c|c|c|c|c|}
				\hline 
				N & $\theta_{0}$ & $\hat{\theta}^{\textrm{MLE}}$ & $\widehat{\textrm{se}}(\hat{\theta}^{\textrm{MLE}})$ & $\widetilde{\textrm{se}}(\hat{\theta}^{\textrm{MLE}})$ & $\textrm{bias}_{\hat{\theta}^{\textrm{MLE}}}$ & $\textrm{RMSE}_{\hat{\theta}^{\textrm{MLE}}}$ & $\textrm{CV}_{\hat{\theta}^{\textrm{MLE}}}$\tabularnewline
				\hline 
				\multirow{3}{*}{$100$} & $0.1$ & $0.132$ & $0.184$ & $0.135$ & $0.032$ & $0.186$ & $1.389$\tabularnewline
				\cline{2-8} \cline{3-8} \cline{4-8} \cline{5-8} \cline{6-8} \cline{7-8} \cline{8-8} 
				& $0.5$ & $0.486$ & $0.134$ & $0.144$ & $-0.014$ & $0.135$ & $0.276$\tabularnewline
				\cline{2-8} \cline{3-8} \cline{4-8} \cline{5-8} \cline{6-8} \cline{7-8} \cline{8-8} 
				& $0.9$ & $0.893$ & $0.051$ & $0.055$ & $-0.007$ & $0.052$ & $0.057$\tabularnewline
				\hline 
				\multirow{3}{*}{$500$} & $0.1$ & $0.100$ & $0.093$ & $0.074$ & $0.000$ & $0.093$ & $0.930$\tabularnewline
				\cline{2-8} \cline{3-8} \cline{4-8} \cline{5-8} \cline{6-8} \cline{7-8} \cline{8-8} 
				& $0.5$ & $0.498$ & $0.058$ & $0.059$ & $-0.002$ & $0.058$ & $0.117$\tabularnewline
				\cline{2-8} \cline{3-8} \cline{4-8} \cline{5-8} \cline{6-8} \cline{7-8} \cline{8-8} 
				& $0.9$ & $0.899$ & $0.022$ & $0.022$ & $-0.001$ & $0.022$ & $0.024$\tabularnewline
				\hline 
				\multirow{3}{*}{$1500$} & $0.1$ & $0.097$ & $0.056$ & $0.050$ & $-0.003$ & $0.056$ & $0.574$\tabularnewline
				\cline{2-8} \cline{3-8} \cline{4-8} \cline{5-8} \cline{6-8} \cline{7-8} \cline{8-8} 
				& $0.5$ & $0.499$ & $0.034$ & $0.034$ & $-0.001$ & $0.034$ & $0.068$\tabularnewline
				\cline{2-8} \cline{3-8} \cline{4-8} \cline{5-8} \cline{6-8} \cline{7-8} \cline{8-8} 
				& $0.9$ & $0.900$ & $0.012$ & $0.012$ & $0.000$ & $0.012$ & $0.014$\tabularnewline
				\hline
			\end{tabular}
		\par\end{centering}
	\end{table}
\par\end{center}				
				
\begin{center}
	\begin{table}
		\caption{Monte Carlo results for the bootstrap estimator with irregularly spaced times. The MCE estimated is $0.005$.}
		\label{IMA-MC-Table-IrregularSpaced-b}
		\begin{centering}
			\begin{tabular}{|c|c|c|c|c|c|c|c|}
				\hline 
				N & $\theta_{0}$ & $\hat{\theta}^{\textrm{b}}$ & $\widehat{\textrm{se}}(\hat{\theta}^{\textrm{b}})$ & $\widetilde{\textrm{se}}(\hat{\theta}^{\textrm{b}})$ & $\textrm{bias}_{\hat{\theta}^{\textrm{b}}}$ & $\textrm{RMSE}_{\hat{\theta}^{\textrm{b}}}$ & $\textrm{CV}_{\hat{\theta}^{\textrm{b}}}$\tabularnewline
				\hline 
				\multirow{3}{*}{$100$} & $0.1$ & $0.167$ & $0.136$ & $0.098$ & $0.067$ & $0.151$ & $0.811$\tabularnewline
				\cline{2-8} \cline{3-8} \cline{4-8} \cline{5-8} \cline{6-8} \cline{7-8} \cline{8-8} 
				& $0.5$ & $0.473$ & $0.140$ & $0.139$ & $-0.027$ & $0.143$ & $0.296$\tabularnewline
				\cline{2-8} \cline{3-8} \cline{4-8} \cline{5-8} \cline{6-8} \cline{7-8} \cline{8-8} 
				& $0.9$ & $0.887$ & $0.055$ & $0.056$ & $-0.013$ & $0.056$ & $0.062$\tabularnewline
				\hline 
				\multirow{3}{*}{$500$} & $0.1$ & $0.111$ & $0.069$ & $0.059$ & $0.011$ & $0.070$ & $0.626$\tabularnewline
				\cline{2-8} \cline{3-8} \cline{4-8} \cline{5-8} \cline{6-8} \cline{7-8} \cline{8-8} 
				& $0.5$ & $0.496$ & $0.059$ & $0.059$ & $-0.004$ & $0.059$ & $0.118$\tabularnewline
				\cline{2-8} \cline{3-8} \cline{4-8} \cline{5-8} \cline{6-8} \cline{7-8} \cline{8-8} 
				& $0.9$ & $0.898$ & $0.022$ & $0.023$ & $-0.002$ & $0.022$ & $0.024$\tabularnewline
				\hline 
				\multirow{3}{*}{$1500$} & $0.1$ & $0.099$ & $0.047$ & $0.043$ & $-0.001$ & $0.047$ & $0.474$\tabularnewline
				\cline{2-8} \cline{3-8} \cline{4-8} \cline{5-8} \cline{6-8} \cline{7-8} \cline{8-8} 
				& $0.5$ & $0.498$ & $0.034$ & $0.033$ & $-0.002$ & $0.034$ & $0.068$\tabularnewline
				\cline{2-8} \cline{3-8} \cline{4-8} \cline{5-8} \cline{6-8} \cline{7-8} \cline{8-8} 
				& $0.9$ & $0.900$ & $0.012$ & $0.012$ & $0.000$ & $0.012$ & $0.014$\tabularnewline
				\hline 
			\end{tabular}
		\par\end{centering}
	\end{table}
\par\end{center}


A second simulation experiment is performed in order to assess the performance of the CARMA$(p,q)$ model in fitting an IMA simulated process. The CARMA models has been generally used to fit irregularly sampled time series. However, as mentioned previously, the CARMA models are restricted to $p>q$, so the IMA model is not a particular case of the CARMA model \citep[see,][]{Phadke1974,Thornton2013}. For this reason, it is interesting to evaluate whether a CARMA model can properly fit an IMA process.

In this experiment, we generate 1000 sequences of the IMA process with two different lengths and four different values of the $\theta$ parameter. The irregular times were generated using a mixture of two exponential distributions with $130$ and $6.5$ as the mean of each exponential distribution and $0.15$ and $0.85$ as their respective weights. We chose this distribution of times in order to obtain large time gaps, which can be an issue when a model assumes continuous time. 

In order to fit the CARMA model we use the function \textit{carma} of the \textit{growth} package of {\em R}. In addition, we opted to use the CARMA$(2,1)$ model since it is the most parsimonious CARMA model with a moving average component that can be fitted. Finally, we used the  Root mean squared error as the goodness-of-fit measure. In table \ref{IMA-CARMA} we present the Monte Carlo results of this experiment. Note that the IMA model consistently outperforms  the CARMA$(2,1)$ model. In particular, for large values of $\theta$ and $N$ the IMA model have almost 100\% of success rate when selecting the model with lowest MSE. This result is relevant because when using the -2 log likelihood as a goodness-of-fit measure we are not penalizing by the number of parameters and anyways the IMA model has better performance than the CARMA model. In Appendix \ref{sec:IMA-CARMA-App}, we uses other goodness-of-fit measures such as the Akaike Criterion and the -2 log likelihood. Considering these criteria, the IMA model also fits the simulated processes better than the CARMA$(2,1)$ model in most cases.

\begin{center}
	\begin{table}
		\caption{Monte Carlo results of the experiment to compare the performance of the IMA and CARMA(2,1) models in fitting IMA model. The goodness of fit measure estimated from both models is the Root Mean Squared Error.}
		\label{IMA-CARMA-App2}
		\begin{centering}
			\begin{tabular}{|c|c|c|c|c|c|c|c|c|}
				\hline 
				N & $\theta_{0}$ & $\hat{\theta}$ & $\widehat{\textrm{se}}(\hat{\theta})$ & MSE\_IMA & SD(MSE\_IMA) & MSE\_C21 & SD(MSE\_C21) &  \% Success  \tabularnewline
				\hline 
				\multirow{4}{*}{$300$}& 0.95 & 0.9437 & 0.0568 & 0.8383 & 0.0440 & 0.9154 & 0.0554 & 0.9880 \tabularnewline
				\cline{2-9} \cline{3-9} \cline{4-9} \cline{5-9} \cline{6-9} \cline{7-9} \cline{8-9} \cline{9-9} 
				& 0.9 & 0.8912 & 0.0534 & 0.8815 & 0.0422 & 0.9272 & 0.0501 & 0.9550 \tabularnewline
				\cline{2-9} \cline{3-9} \cline{4-9} \cline{5-9} \cline{6-9} \cline{7-9} \cline{8-9} \cline{9-9}  
				& 0.7 & 0.6762 & 0.1171 & 0.9277 & 0.0398 & 0.9432 & 0.0421 & 0.8240 \tabularnewline
								\cline{2-9} \cline{3-9} \cline{4-9} \cline{5-9} \cline{6-9} \cline{7-9} \cline{8-9} \cline{9-9} 
				& 0.5 & 0.4900 & 0.1346 & 0.9427 & 0.0381 & 0.9496 & 0.0416 & 0.6220 \tabularnewline
				\hline 
				\multirow{4}{*}{$1000$} & 0.95 & 0.9464 & 0.0473 & 0.8385 & 0.0336 & 0.9200 & 0.0497 & 0.9970 \tabularnewline
				\cline{2-9} \cline{3-9} \cline{4-9} \cline{5-9} \cline{6-9} \cline{7-9} \cline{8-9} \cline{9-9} 
				& 0.9  & 0.8971 & 0.0399 & 0.8818 & 0.0332 & 0.9289 & 0.0435 & 0.9970 \tabularnewline
				\cline{2-9} \cline{3-9} \cline{4-9} \cline{5-9} \cline{6-9} \cline{7-9} \cline{8-9} \cline{9-9} 
				 & 0.7 & 0.6924 & 0.0675 & 0.9311 & 0.0329 & 0.9456 & 0.0355 & 0.9620 \tabularnewline
				\cline{2-9} \cline{3-9} \cline{4-9} \cline{5-9} \cline{6-9} \cline{7-9} \cline{8-9} \cline{9-9} 
				& 0.5 & 0.4949 & 0.0737 & 0.9463 & 0.0324 & 0.9541 & 0.0340 & 0.8980  \tabularnewline
				\hline 
			\end{tabular}
		\par\end{centering}
	\end{table}
\par\end{center}

\subsection{Non-Gaussian errors}

There are many applications where Gaussian assumption is misspecified, but it is assumed in the likelihood to obtain the ML estimators. These estimators are commonly known as quasi-ML estimators or QMLE.

This section shows the performance measures for the QMLE in a Monte Carlo experiment where it is simulated trajectories of the IMA model with non-Gaussian errors exploiting its constructionist viewpoint. First, it is used a Student distribution with shape parameter 7 (degree of freedom). Second, a Generalized error distribution with shape parameter 1.28 (heavy tails) is used. Both distribution are standardized according to specifications used by \cite{Ghalanos2020}.

Tables \ref{IMA-MC-Table-IrregularSpaced-std-QMLE} and \ref{IMA-MC-Table-IrregularSpaced-ged-QMLE} display the results obtained for Student and Generalized errors, respectively. In both cases, the performance obtained for the QMLE is very similar to the performance obtained for the MLE presented in Table 1 (Gaussian case). This suggests that QML estimators are robust to deviations from Gaussianity.

\begin{center}
	\begin{table}
		\caption{Monte Carlo results for the QML estimator with irregularly spaced times and Student errors with shape parameter $7$. The MCE estimated is $0.004$.}
		\label{IMA-MC-Table-IrregularSpaced-std-QMLE}
		\begin{centering}
			\begin{tabular}{|c|c|c|c|c|c|c|c|}
				\hline 
				N & $\theta_{0}$ & $\hat{\theta}^{\textrm{MLE}}$ & $\widehat{\textrm{se}}(\hat{\theta}^{\textrm{MLE}})$ & $\widetilde{\textrm{se}}(\hat{\theta}^{\textrm{MLE}})$ & $\textrm{bias}_{\hat{\theta}^{\textrm{MLE}}}$ & $\textrm{RMSE}_{\hat{\theta}^{\textrm{MLE}}}$ & $\textrm{CV}_{\hat{\theta}^{\textrm{MLE}}}$\tabularnewline
				\hline          
				\multirow{3}{*}{$100$} & $0.1$ & $0.126$ & $0.183$ & $0.133$ & $0.026$ & $0.185$ & $1.452$\tabularnewline
				\cline{2-8} \cline{3-8} \cline{4-8} \cline{5-8} \cline{6-8} \cline{7-8} \cline{8-8}   
				& $0.5$ & $0.488$ & $0.138$ & $0.140$ & $-0.012$ & $0.138$ & $0.283$\tabularnewline
				\cline{2-8} \cline{3-8} \cline{4-8} \cline{5-8} \cline{6-8} \cline{7-8} \cline{8-8}          
				& $0.9$ & $0.897$ & $0.050$ & $0.054$ & $-0.003$ & $0.050$ & $0.055$\tabularnewline
				\hline        
				\multirow{3}{*}{$500$} & $0.1$ & $0.098$ & $0.099$ & $0.075$ & $-0.002$ & $0.099$ & $1.012$\tabularnewline
				\cline{2-8} \cline{3-8} \cline{4-8} \cline{5-8} \cline{6-8} \cline{7-8} \cline{8-8}          
				& $0.5$ & $0.495$ & $0.060$ & $0.063$ & $-0.005$ & $0.060$ & $0.122$\tabularnewline
				\cline{2-8} \cline{3-8} \cline{4-8} \cline{5-8} \cline{6-8} \cline{7-8} \cline{8-8}       
				& $0.9$ & $0.899$ & $0.022$ & $0.022$ & $-0.001$ & $0.022$ & $0.024$\tabularnewline
				\hline           
				\multirow{3}{*}{$1500$} & $0.1$ & $0.097$ & $0.056$ & $0.050$ & $-0.003$ & $0.056$ & $0.575$\tabularnewline
				\cline{2-8} \cline{3-8} \cline{4-8} \cline{5-8} \cline{6-8} \cline{7-8} \cline{8-8}           
				& $0.5$ & $0.498$ & $0.034$ & $0.035$ & $-0.002$ & $0.034$ & $0.069$\tabularnewline
				\cline{2-8} \cline{3-8} \cline{4-8} \cline{5-8} \cline{6-8} \cline{7-8} \cline{8-8}          
				& $0.9$ & $0.900$ & $0.012$ & $0.013$ & $0.000$ & $0.012$ & $0.014$\tabularnewline
				\hline
			\end{tabular}
		\par\end{centering}
	\end{table}
\par\end{center}

\begin{center}
	\begin{table}
		\caption{Monte Carlo results for the QML estimator with irregularly spaced times and Generalized errors  with shape parameter $1.28$. The MCE estimated is $0.004$.}
		\label{IMA-MC-Table-IrregularSpaced-ged-QMLE}
		\begin{centering}
			\begin{tabular}{|c|c|c|c|c|c|c|c|}
				\hline 
				N & $\theta_{0}$ & $\hat{\theta}^{\textrm{MLE}}$ & $\widehat{\textrm{se}}(\hat{\theta}^{\textrm{MLE}})$ & $\widetilde{\textrm{se}}(\hat{\theta}^{\textrm{MLE}})$ & $\textrm{bias}_{\hat{\theta}^{\textrm{MLE}}}$ & $\textrm{RMSE}_{\hat{\theta}^{\textrm{MLE}}}$ & $\textrm{CV}_{\hat{\theta}^{\textrm{MLE}}}$\tabularnewline
				\hline                
				\multirow{3}{*}{$100$} & $0.1$ & $0.128$ & $0.193$ & $0.135$ & $0.028$ & $0.195$ & $1.503$\tabularnewline
				\cline{2-8} \cline{3-8} \cline{4-8} \cline{5-8} \cline{6-8} \cline{7-8} \cline{8-8}          
				& $0.5$ & $0.491$ & $0.136$ & $0.139$ & $-0.009$ & $0.136$ & $0.276$\tabularnewline
				\cline{2-8} \cline{3-8} \cline{4-8} \cline{5-8} \cline{6-8} \cline{7-8} \cline{8-8}                    
				& $0.9$ & $0.896$ & $0.050$ & $0.053$ & $-0.004$ & $0.051$ & $0.056$\tabularnewline
				\hline                   
				\multirow{3}{*}{$500$} & $0.1$ & $0.102$ & $0.096$ & $0.077$ & $0.002$ & $0.096$ & $0.945$\tabularnewline
				\cline{2-8} \cline{3-8} \cline{4-8} \cline{5-8} \cline{6-8} \cline{7-8} \cline{8-8}                   
				& $0.5$ & $0.497$ & $0.060$ & $0.062$ & $-0.003$ & $0.060$ & $0.120$\tabularnewline
				\cline{2-8} \cline{3-8} \cline{4-8} \cline{5-8} \cline{6-8} \cline{7-8} \cline{8-8}           
				& $0.9$ & $0.899$ & $0.022$ & $0.021$ & $-0.001$ & $0.022$ & $0.024$\tabularnewline
				\hline                   
				\multirow{3}{*}{$1500$} & $0.1$ & $0.098$ & $0.055$ & $0.050$ & $-0.002$ & $0.055$ & $0.561$\tabularnewline
				\cline{2-8} \cline{3-8} \cline{4-8} \cline{5-8} \cline{6-8} \cline{7-8} \cline{8-8}                
				& $0.5$ & $0.500$ & $0.034$ & $0.034$ & $0.000$ & $0.034$ & $0.068$\tabularnewline
				\cline{2-8} \cline{3-8} \cline{4-8} \cline{5-8} \cline{6-8} \cline{7-8} \cline{8-8}                 
				& $0.9$ & $0.900$ & $0.012$ & $0.012$ & $0.000$ & $0.012$ & $0.014$\tabularnewline
				\hline
			\end{tabular}
			\par\end{centering}
	\end{table}
	\par\end{center}

\section{Applications\label{sec:IMA-Application}}
This section illustrates the application of the proposed time series methodology to two real-life datasets. The first example is concerned with medical data whereas in the second application is described the analysis of light curves in astronomy.

\subsection{Lung function of an asthma patient}
\citet{Belcher1994} analyzed measurements of the lung function of an asthma patient. The  observations are collected mostly at 2 hour time intervals but with irregular gaps (see the unequal spaced of tick marks in Figure \ref{medical-application}). However, as it was shown in \citet{Wang2013}, the trend component (obtained by decomposing original time series into trend, seasonal, and irregular components via the Kalman smoother) exhibits structural changes after 100$th$ observation. Thus,  the first 100 observations are considered here to analyze such a phenomenon. Below,  the ML and bootstrap estimates are reported along with  their respective estimated standard errors. Note that the estimates are significant at the 5\% significance level.
\[
\hat{\theta}^{\textrm{MLE}}=0.853\quad\widehat{\textrm{se}}(\hat{\theta}^{\textrm{MLE}})=0.069\quad\hat{\sigma}_{\textrm{MLE}}^{2}=258.286\quad\widehat{\textrm{se}}(\hat{\sigma}_{\textrm{MLE}}^{2})=36.537
\]

\[
\hat{\theta}^{\textrm{b}}=0.841\quad\widehat{\textrm{se}}(\hat{\theta}^{\textrm{b}})=0.077\quad\hat{\sigma}_{\textrm{b}}^{2}=259.270\quad\widehat{\textrm{se}}(\hat{\sigma}_{\textrm{b}}^{2})=32.662
\]

\begin{figure}
\begin{centering}
\includegraphics[scale=0.75]{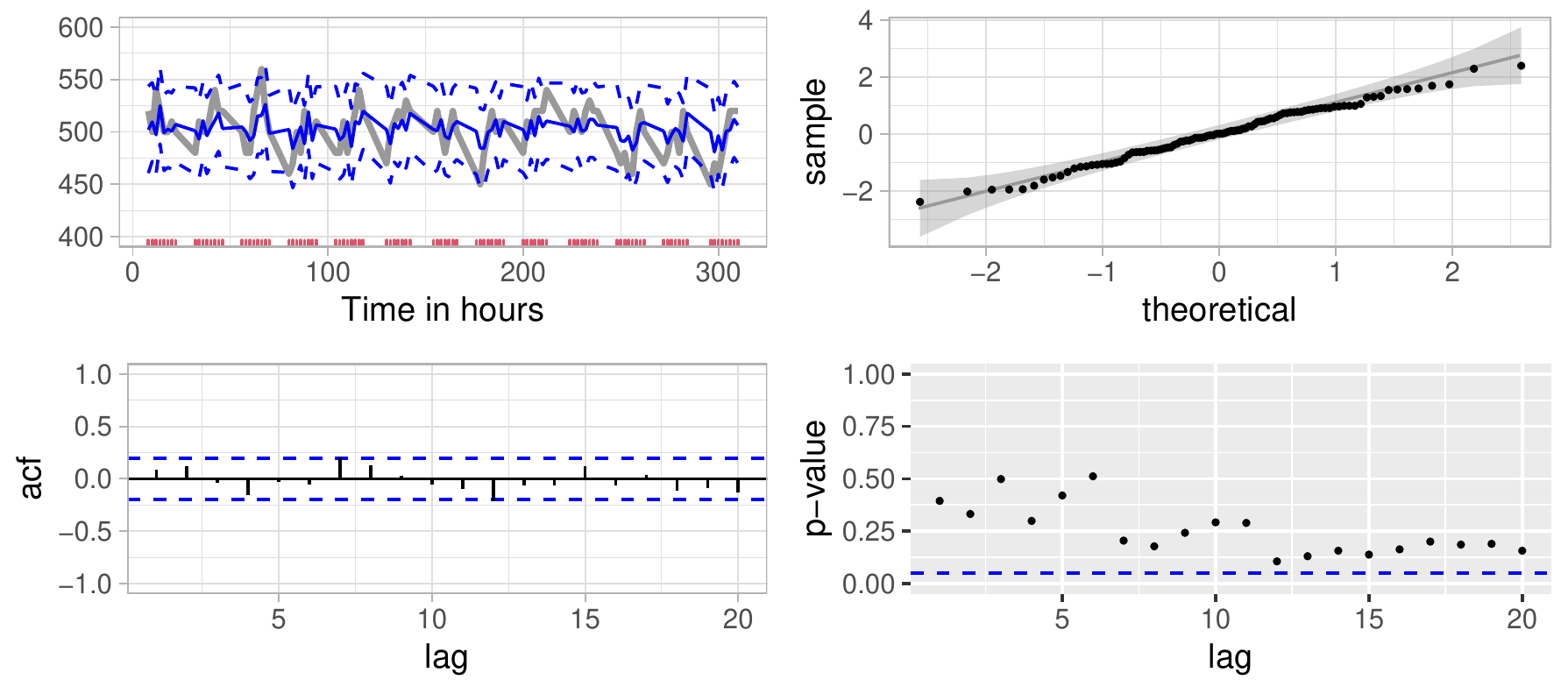}
\end{centering}
\caption{On the left-top, the lung function of an asthma patient with the predicted values and their respective variability bands. On the right-top, quantile-quantile plot of the standardized residuals with normality reference bands. In this case, we use $\hat{\theta}^{\textrm{MLE}}$. On the bottom-left, the autocorrelation function estimated of the standardized residuals. On the bottom-right, the Ljung-Box test for randomness of the standardized residuals.}
\label{medical-application}
\end{figure}

From Figure \ref{medical-application}, the fit seems adequate. Also, the standardized residuals seem to follow a standard normal distribution \citep{Nair1982}. Figure \ref{medical-application} shows the ACF estimated and the results from a Ljung-Box test for the standardized residuals. Observe that the residuals satisfy the white noise test at the $5$\% significance level. Note that, since the standardized residuals are assumed  to be realizations of a random sample, its correlation structure does not depend on the irregularly spaced between observations. Thus, unlike the original time series, the  ACF and the Ljung-Box test can be applied to the standardized residuals.

Estimated mean squared errors for the one-step predictors are presented in Figure \ref{application-MSE} under the IMA and the MA model. Notice that under the IMA model, the MSE is smaller when the time distances become smaller while, under the MA model, the MSE is constant regardless of the size of these distances. Also, it is important to consider that both MSE, on average, exhibit similar behavior.

\begin{figure}
\begin{centering}
\includegraphics[height=2.5in,width=4.5in]{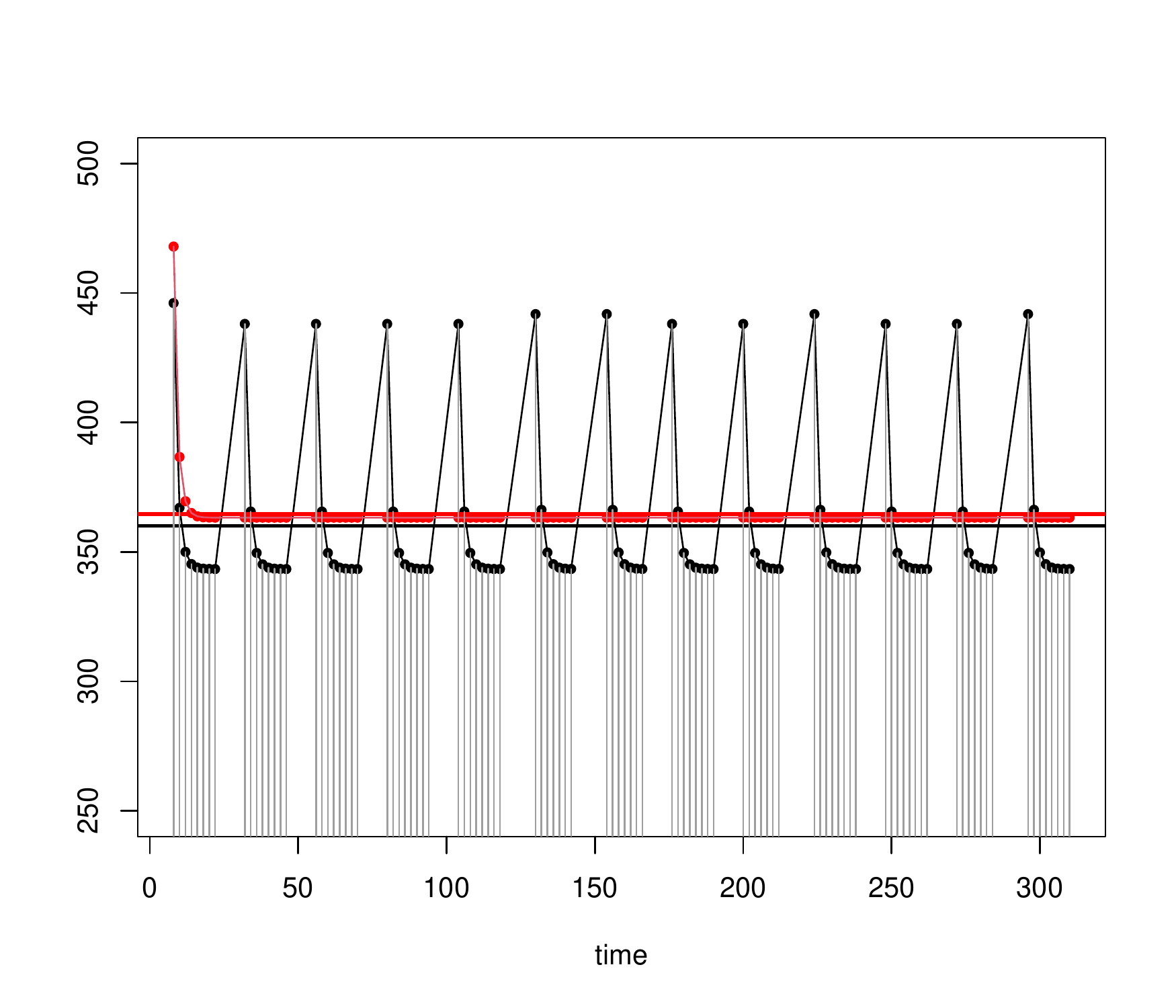}
\end{centering}
\caption{Estimated mean squared errors for the lung function data. Black line represent MSE under the IMA model while red line represent MSE under the MA model. Horizontal lines are their corresponding averages.}
\label{application-MSE}
\end{figure}

\subsection{Light curves}

In astronomy it is common to find time series measured at irregular times. These temporal data generally correspond to the brightness of an astronomical object over time. The time sequence of measurements of the brightness of an astronomical object is known as the light curve of this object. In this example, the IMA model is fitted to the light curve of a Blazar object. Blazar objects are characterized by stochastics signals in its light curves. The light curve used in this example was observed with the Zwicky Transient Facility (ZTF, \citet{Bellm_2018_1}) survey (coded as ZTF18abvfmot). The observations from the ZTF survey were processed by the ALeRCE broker \citep{ALeRCE}. The light curve has 105 observations taken over a range of approximately 480 days. Before fitting the model, we perform a transformation of the time series in order to stabilize its variance. Hereafter,  the Blazar time series will be referred to as the light curve.\\ 

The IMA model parameters $\theta$ and $\sigma^2$ estimated in the Blazar light curve by the maximum likelihood and bootstrap methods, and their respective estimated standard errors are the following,

\[
\hat{\theta}^{\textrm{MLE}}=0.815\quad\widehat{\textrm{se}}(\hat{\theta}^{\textrm{MLE}})=0.085\quad\hat{\sigma}_{\textrm{MLE}}^{2}=0.593 \quad\widehat{\textrm{se}}(\hat{\sigma}_{\textrm{MLE}}^{2})=0.084 
\]

\[
\hat{\theta}^{\textrm{b}}=0.802\quad\widehat{\textrm{se}}(\hat{\theta}^{\textrm{b}})=0.089\quad\hat{\sigma}_{\textrm{b}}^{2}=0.581\quad\widehat{\textrm{se}}(\hat{\sigma}_{\textrm{b}}^{2})=0.058
\]

Note that, the parameter $\theta$ estimated by both methods is significant greater than zero at the 5\% significance level. Consequently, the IMA model detect a significant correlation structure in the Blazar ligth curve. To assess whether the model is able to explain the entire correlation structure of the data, a residual analysis is performed. Figure \ref{fig:IMAResidualsLC} shows the residuals of the model. According to the figure, residuals are normally distributed and are uncorrelated. Therefore, the IMA model here removes the serial autocorrelation.\\

\begin{figure}
\centering 
\includegraphics[height=3in,width=5.5in]{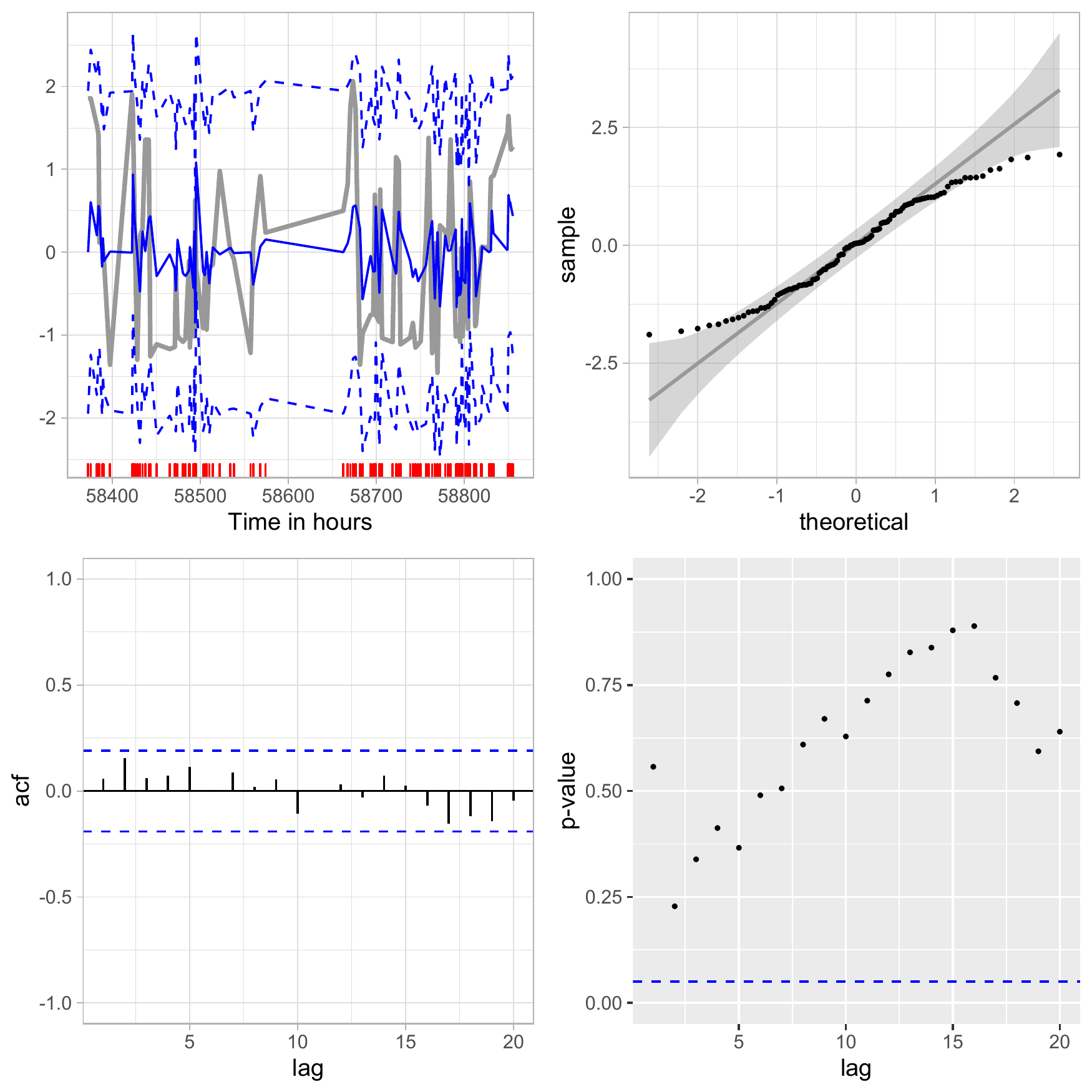} 
\caption{Analysis of the residuals of the IMA model fitted to the Blazar light curve. \label{fig:IMAResidualsLC}} 
\end{figure} 

To assess the goodness of fit of the IMA model on the Blazar light curve, we estimate its mean square errors for the one-step predictors. We compare these estimates with the mean square errors obtained for the regular MA model. Note from Figure \ref{MSE_LC1} that the average of the MSE of both models are similar. However, as in the previous example, the MSE of the IMA model is smaller of the obtained with the MA model when the time gaps are smaller, while for large gaps the MA model have a smaller value of the MSE.

\begin{figure}
\begin{centering}
\includegraphics[height=2.5in,width=4.5in]{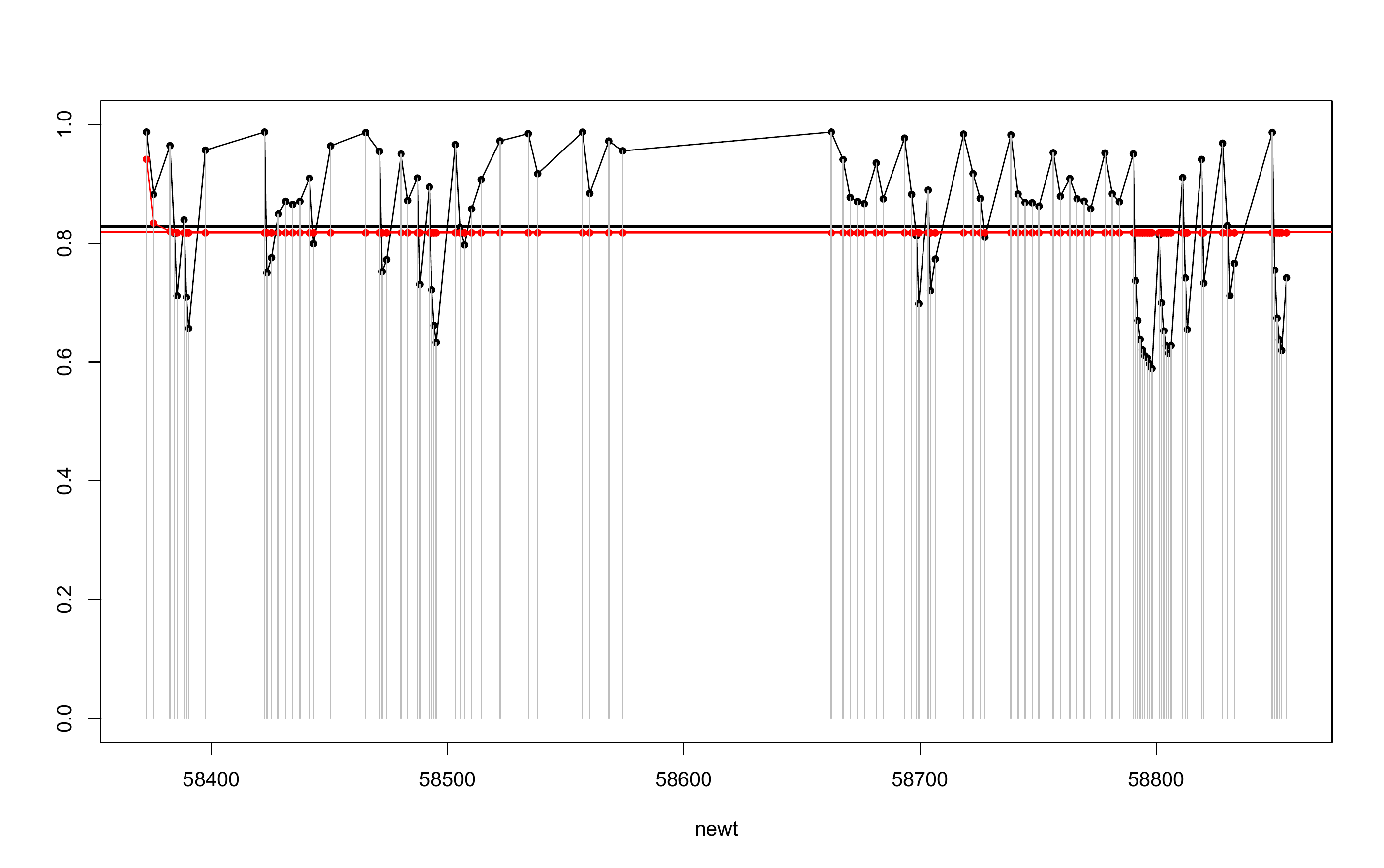}
\end{centering}
\caption{Estimated mean squared errors for the Blazar light curve. Black line represent MSE under the IMA model while red line represent MSE under the MA model. Horizontal lines are their corresponding averages.}
\label{MSE_LC1}
\end{figure}

\section{Conclusions\label{conclusions}}
This paper proposes an irregularly spaced first-order MA model that allows for the handling of first-order moving averages structures with irregularly spaced times. Its formal definition is provided and some of its statistical properties are studied.  State space representations along with one-step linear predictors are also provided. Furthermore, the finite-sample performance of the  proposed estimation methods is investigated by means of Monte Carlo simulations for both Gaussian and non-Gaussian distributions errors. The proposed methods display very good estimation performance in all the cases investigated. The performance of the proposed IMA model is also compared to the CARMA processes. Finally, the practical application of the proposed methodology is illustrated by means of two real-life data examples involving medical and astronomical time series.

\section*{Acknowledgments}
The first author was supported by CONICYTPFCHA/2015-21151457. The authors WP, SE and FE was supported from the ANID Millennium Science Initiative ICN12\_009, awarded to the Millennium Institute of Astrophysics

\appendix

\section{Constructionist viewpoint\label{sec:Constructionist-viewpoint}}

Details about defining an IMA model from the constructionist viewpoint are discussed next. The main idea is to specify the IMA process as a function of other (often simpler) stochastic processes. Let $\{\zeta_{t_{n}}\}_{n\geq1}$ be uncorrelated random variables each with mean 0 and variance 1. Now, consider\\
\begin{gather*}
X_{t_{1}}=\nu_{1}^{\nicefrac{1}{2}}\zeta_{t_{1}},\\
X_{t_{n+1}}=\nu_{n+1}^{\nicefrac{1}{2}}\zeta_{t_{n+1}}+\omega_{n}\nu_{n}^{\nicefrac{1}{2}}\zeta_{t_{n}},\quad\textrm{for}\;n\geq1,
\end{gather*}
\\
where $\{\nu_{n}\}_{n\geq1}$ and $\{\omega_{n}\}_{n\geq1}$ are time-varying
sequences that characterize the moments of the process. Thus, for
$n\geq1$, we have $\textrm{E}(X_{t_{n}})=0$,\\
\begin{gather*}
\textrm{Var}(X_{t_{1}})=\nu_{1},\;\textrm{Var}(X_{t_{n+1}})=\nu_{n+1}+\omega_{n}^{2}\nu_{n},\;\textrm{and}\\
\textrm{Cov}(X_{t_{n}},X_{t_{n+k}})=\begin{cases}
\omega_{n}\nu_{n}, & k=1,\\
0, & k\geq2.
\end{cases}
\end{gather*}

The main goal is to find $\{\nu_{n}\}_{n\geq1}$ and $\{\omega_{n}\}_{n\geq1}$
so that $\{X_{t_{n}}\}_{n\geq1}$ be a stationary process. For this,
we need that, for $n\geq1$, $\textrm{Var}(X_{t_{n+1}})=\textrm{Var}(X_{t_{1}})=\gamma_{0}$
and $\textrm{Cov}(X_{t_{n}},X_{t_{n+1}})=\gamma_{1,\Delta_{n+1}}$
with $\gamma_{1,\Delta_{n+1}}$ a function of $\Delta_{n+1}=t_{n+1}-t_{n}$.
Thus,\\
\begin{gather*}
\nu_{n+1}+\omega_{n}^{2}\nu_{n}=\nu_{1}=\gamma_{0}\;\textrm{and}\\
\omega_{n}\nu_{n}=\gamma_{1,\Delta_{n+1}}\quad\textrm{for}\;n\geq1.
\end{gather*}
\\
From these equations, we obtain
\begin{equation}
    \omega_{n}=\frac{\gamma_{1,\Delta_{n+1}}}{\nu_{n}}\quad\textrm{and}\quad\nu_{n+1}=\gamma_{0}-\frac{\gamma_{1,\Delta_{n+1}}^{2}}{\nu_{n}},\quad\textrm{with}\;\nu_{1}=\gamma_{0}.\label{eq:genbackconfrac}
\end{equation}
\\
Therefore, we can set a real-valued stationary process defining $\gamma_{0}$
and $\{\gamma_{1,\Delta_{n+1}}\}_{n\geq1}$ suitably, that is, in
such a way that $\nu_{n}>0$, for all $n$. Proposition \ref{prop1} shows that setting
\[
\gamma_{0}=\sigma^{2}(1+\theta^{2})\qquad\textrm{and}\qquad\gamma_{1,\Delta_{n+1}}=\sigma^{2}\theta^{\Delta_{n+1}}\quad\textrm{for}\;n\geq1,
\]
with $\sigma^{2}>0$ and $0\leq\theta<1$, the general backward continued fraction $\{\nu_{n}\}_{n\geq1}$ is a strictly positive sequence. Hence, $\{X_{t_{n}}\}_{n\geq1}$
is a well defined real-valued stationary stochastic process.

\begin{prop}
Consider a real-valued sequence $\{\nu_{n}\}_{n\geq1}$ as defined in \eqref{eq:genbackconfrac}. If $\gamma_{0}>0$, and for $n\geq1$, $\left(\frac{\gamma_{1,\Delta_{n+1}}}{\gamma_{0}}\right)^{2}\leq\nicefrac{1}{4}$
with $\gamma_{1,\Delta_{n+1}}\neq0$, then $\{\nu_{n}\}_{n\geq1}$
is a strictly positive sequence.\label{prop1}
\end{prop}

\begin{proof}
The $(n+1)$-th convergent of the general
backward continued fraction, $\{\nu_{n}\}_{n\geq1}$, is defined as\\
\[
\nu_{n+1}=\left[\gamma_{0}+\frac{-\gamma_{1,\Delta_{2}}^{2}}{\gamma_{0}+}\frac{-\gamma_{1,\Delta_{3}}^{2}}{\gamma_{0}+}\cdots\frac{-\gamma_{1,\Delta_{n+1}}^{2}}{\gamma_{0}}\right]_{b}=\gamma_{0}+\cfrac{-\gamma_{1,\Delta_{n+1}}^{2}}{\gamma_{0}+\cfrac{-\gamma_{1,\Delta_{n}}^{2}}{\gamma_{0}+\cfrac{\ddots}{\gamma_{0}+\cfrac{-\gamma_{1,\Delta_{2}}^{2}}{\gamma_{0}}}}}.
\]

Thus, few convergent are\\
{\footnotesize{}
\begin{align*}
\nu_{2} & =\left[\gamma_{0}+\frac{-\gamma_{1,\Delta_{2}}^{2}}{\gamma_{0}+}\right]_{b}=\frac{\gamma_{0}^{2}-\gamma_{1,\Delta_{2}}^{2}}{\gamma_{0}},\\
\nu_{3} & =\left[\gamma_{0}+\frac{-\gamma_{1,\Delta_{2}}^{2}}{\gamma_{0}+}\frac{-\gamma_{1,\Delta_{3}}^{2}}{\gamma_{0}+}\right]_{b}=\frac{\gamma_{0}^{3}-\gamma_{0}\gamma_{1,\Delta_{2}}^{2}-\gamma_{0}\gamma_{1,\Delta_{3}}^{2}}{\gamma_{0}^{2}-\gamma_{1,\Delta_{2}}^{2}},\\
\nu_{4} & =\left[\gamma_{0}+\frac{-\gamma_{1,\Delta_{2}}^{2}}{\gamma_{0}+}\frac{-\gamma_{1,\Delta_{3}}^{2}}{\gamma_{0}+}\frac{-\gamma_{1,\Delta_{4}}^{2}}{\gamma_{0}+}\right]_{b}=\frac{\gamma_{0}^{4}-\gamma_{0}^{2}\gamma_{1,\Delta_{2}}^{2}-\gamma_{0}^{2}\gamma_{1,\Delta_{3}}^{2}-\gamma_{0}^{2}\gamma_{1,\Delta_{4}}^{2}+\gamma_{1,\Delta_{2}}^{2}\gamma_{1,\Delta_{4}}^{2}}{\gamma_{0}^{3}-\gamma_{0}\gamma_{1,\Delta_{2}}^{2}-\gamma_{0}\gamma_{1,\Delta_{3}}^{2}}.
\end{align*}
}{\footnotesize\par}

The ratio for the general backward continued fraction, $\{\nu_{n}\}_{n\geq1}$, is\\
\[
\nu_{1}=\frac{P_{1}}{P_{0}}\;\textrm{and}\;\nu_{n+1}=\frac{P_{n+1}}{P_{n}},
\]
\\
for $n\geq1$, where the sequence $\{P_{n}\}_{n\geq0}$ is\\
\[
P_{n+1}=\gamma_{0}P_{n}-\gamma_{1,\Delta_{n+1}}^{2}P_{n-1}
\]
\\
with $P_{1}=\gamma_{0}$ and $P_{0}=1$. This sequence is obtained
by what is known as the Wallis-Euler recurrence relations \citep{Loya2017}.
\cite{El-Mikkawy2006} proves that $P_{k}>0$, for $k=1,\ldots,n+1$,
if and only if the matrix\\
\[
\boldsymbol{\Gamma}_{n+1}=\begin{bmatrix}\gamma_{0} & \gamma_{1,\Delta_{2}} & \cdots & 0 & 0\\
\gamma_{1,\Delta_{2}} & \gamma_{0} & \cdots & 0 & 0\\
\vdots & \vdots & \ddots & \vdots & \vdots\\
0 & 0 & \cdots & \gamma_{0} & \gamma_{1,\Delta_{n+1}}\\
0 & 0 & \cdots & \gamma_{1,\Delta_{n+1}} & \gamma_{0}
\end{bmatrix}
\]
\\
is positive definite with $\boldsymbol{\Gamma}_{1}=[\gamma_{0}]$.
Further, they proved that $\det\boldsymbol{\Gamma}_{1}=P_{1}$ and
$\det\boldsymbol{\Gamma}_{n+1}=P_{n+1}$.

Thus, from \eqref{eq:ProcessConditions}, if $\gamma_{0}>0$
and, for $n\geq1$, $\left(\frac{\gamma_{1,\Delta_{n+1}}}{\gamma_{0}}\right)^{2}\leq\nicefrac{1}{4}$
with $\gamma_{1,\Delta_{n+1}}\neq0$, then $\{\nu_{n}\}_{n\geq1}$
is a strictly positive sequence.
\end{proof}

\section{Proof of $1<c_{n}(\theta)<2$ for $n\geq1$\label{sec:On-gbcf-IMA}}

Consider the general backward continued fraction of the IMA model
$$c_{n}(\theta)=1+\theta^{2}-\frac{\theta^{2\Delta_{n}}}{c_{n-1}(\theta)},\quad n\geq2,$$
with $c_{1}(\theta)=1+\theta^{2}$.

\begin{prop}
If $0\leq\theta<1$, and $\Delta_{n}\geq1$, for all $n$, then
$$1\leq c_{n}(\theta)<2\quad\textrm{for}\quad n\geq1.$$
\end{prop}

\begin{proof}
By induction. For $n=1$, we have straightly that $1\leq c_{1}(\theta)<2$. Now, consider (for $n\geq2$) that $1\leq c_{n-1}(\theta)<2$ as the hypothesis. Then,
$$c_{n}(\theta)=1+\theta^{2}-\frac{\theta^{2\Delta_{n}}}{c_{n-1}(\theta)}<1+\theta^{2}=c_{1}(\theta)<2,$$
since ${\theta^{2\Delta_{n}}}/{c_{n-1}(\theta)}>0$. Also, since ${\theta^{2\Delta_{n}}}/{c_{n-1}(\theta)}<\theta^{2\Delta_{n}}$, then
$$c_{n}(\theta)=1+\theta^{2}-\frac{\theta^{2\Delta_{n}}}{c_{n-1}(\theta)}>1+\theta^{2}-\theta^{2\Delta_{n}}$$
but $\theta^{2\Delta_{n}}\leq\theta^{2}$ since $\Delta_{n}\geq1$. Thus,
$$c_{n}(\theta)>1+\theta^{2}-\theta^{2\Delta_{n}}\geq1$$
obtaining the desired result.
\end{proof}

\section{Innovations algorithm\label{sec:Innovations-algorithm}}

Following  \cite[Proposition 5.2.2]{Brockwell1991}, we
obtain the coefficients of $\hat{X}_{t_{n+1}}=\sum_{j=1}^{n}\theta_{nj}(X_{t_{n+1-j}}-\hat{X}_{t_{n+1-j}})$
through the next proposition.
\begin{prop}
If $\{X_{t_{n}}\}_{n\geq1}$ has zero mean and $\textrm{E}(X_{t_{i}}X_{t_{j}})=\gamma(t_{i},t_{j})$,
where the matrix $\left[\gamma(t_{i},t_{j})\right]_{i,j=1}^{n}$ is
non-singular for each $n=1,2,\ldots$, then the one-step predictors
$\hat{X}_{t_{n+1}}$, and their mean squared errors $\upsilon_{n+1}$,
$n\geq0$, are given by\\
\[
\hat{X}_{t_{n+1}}=\begin{cases}
0 & \textrm{if}\qquad n=0,\\
\sum_{j=1}^{n}\theta_{nj}\left(X_{t_{n+1-j}}-\hat{X}_{t_{n+1-j}}\right) & \textrm{if}\qquad n\geq1,
\end{cases}
\]
\\
and\\
\begin{align*}
\begin{cases}
\upsilon_{1} & =\gamma(t_{1},t_{1}),\\
\theta_{n,n-k} & =\upsilon_{k+1}^{-1}\left(\gamma(t_{n+1},t_{k+1})-\sum_{j=0}^{k-1}\theta_{k,k-j}\theta_{n,n-j}\upsilon_{j+1}\right), k=0,1,\ldots,n-1,\\
\upsilon_{n+1} & =\gamma(t_{n+1},t_{n+1})-\sum_{j=0}^{n-1}\theta_{n,n-j}^{2}\upsilon_{j+1}.
\end{cases}
\end{align*}
\end{prop}
For the irregularly spaced first-order moving average process of general
form we have,\\
\begin{align*}
\theta_{n,j} & =0,\quad2\leq j\leq n,\\
\theta_{n,1} & =\upsilon_{n}^{-1}\gamma_{1,\Delta_{n+1}},\\
\upsilon_{n+1} & =\gamma_{0}-\theta_{n,1}^{2}\upsilon_{n}=\gamma_{0}-\frac{\gamma_{1,\Delta_{n+1}}^{2}}{\upsilon_{n}},\;\textrm{and}\;\upsilon_{1}=\gamma_{0}.
\end{align*}

\section{Measures of performance-regularly spaced case\label{sec:IMA-MC-RegularySpacedTimes}}

In this section, the performance of irregularly spaced case for the ML and Bootstrap estimators is compared in the context of regularly spaced times.
Tables \ref{IMA-MC-Table-RegularSpaced-MLE} and \ref{IMA-MC-Table-RegularSpaced-b} show the Monte Carlo
results. 

In addition, Table \ref{IMA-MC-Table-RegularSpaced-MLE}, also  consider the asymptotic standard error for the ML estimator defined as the square root of
\[
\textrm{se}^{2}(\hat{\theta}^{\textrm{MLE}})=\frac{1-\theta^{2}}{\textrm{N}}.
\]
From these tables, notice that the Monte Carlo results for
the regularly spaced time case exhibit a similar behavior as compared to the irregularly spaced time case. 


\begin{center}
	\begin{table}
		\caption{Monte Carlo results for the ML estimator with regularly spaced times. The MCE estimated is $0.003$.}
		\label{IMA-MC-Table-RegularSpaced-MLE}
		\begin{centering}
			\begin{tabular}{|c|c|c|c|c|c|c|c|c|}
				\hline 
				N & $\theta_{0}$ & $\hat{\theta}^{\textrm{MLE}}$ & $\widehat{\textrm{se}}(\hat{\theta}^{\textrm{MLE}})$ & $\widetilde{\textrm{se}}(\hat{\theta}^{\textrm{MLE}})$ & $\textrm{se}(\hat{\theta}^{\textrm{MLE}})$ & $\textrm{bias}_{\hat{\theta}^{\textrm{MLE}}}$ & $\textrm{RMSE}_{\hat{\theta}^{\textrm{MLE}}}$ & $\textrm{CV}_{\hat{\theta}^{\textrm{MLE}}}$\tabularnewline
				\hline 
				\multirow{3}{*}{$100$} & $0.1$ & $0.108$ & $0.103$ & $0.087$ & $0.099$ & $0.008$ & $0.103$ & $0.952$\tabularnewline
				\cline{2-8} \cline{3-8} \cline{4-8} \cline{5-8} \cline{6-8} \cline{7-8} \cline{8-8} 
				& $0.5$ & $0.502$ & $0.089$ & $0.096$ & $0.087$ & $0.002$ & $0.089$ & $0.178$\tabularnewline
				\cline{2-8} \cline{3-8} \cline{4-8} \cline{5-8} \cline{6-8} \cline{7-8} \cline{8-8} 
				& $0.9$ & $0.910$ & $0.051$ & $0.055$ & $0.044$ & $0.010$ & $0.052$ & $0.056$\tabularnewline
				\hline 
				\multirow{3}{*}{$500$} & $0.1$ & $0.097$ & $0.045$ & $0.044$ & $0.044$ & $-0.003$ & $0.045$ & $0.461$\tabularnewline
				\cline{2-8} \cline{3-8} \cline{4-8} \cline{5-8} \cline{6-8} \cline{7-8} \cline{8-8} 
				& $0.5$ & $0.499$ & $0.039$ & $0.039$ & $0.039$ & $-0.001$ & $0.039$ & $0.078$\tabularnewline
				\cline{2-8} \cline{3-8} \cline{4-8} \cline{5-8} \cline{6-8} \cline{7-8} \cline{8-8} 
				& $0.9$ & $0.901$ & $0.020$ & $0.021$ & $0.019$ & $0.001$ & $0.020$ & $0.022$\tabularnewline
				\hline 
				\multirow{3}{*}{$1500$} & $0.1$ & $0.099$ & $0.026$ & $0.026$ & $0.026$ & $-0.001$ & $0.026$ & $0.261$\tabularnewline
				\cline{2-8} \cline{3-8} \cline{4-8} \cline{5-8} \cline{6-8} \cline{7-8} \cline{8-8} 
				& $0.5$ & $0.500$ & $0.022$ & $0.022$ & $0.022$ & $0.000$ & $0.022$ & $0.045$\tabularnewline
				\cline{2-8} \cline{3-8} \cline{4-8} \cline{5-8} \cline{6-8} \cline{7-8} \cline{8-8} 
				& $0.9$ & $0.901$ & $0.011$ & $0.011$ & $0.011$ & $0.001$ & $0.011$ & $0.013$\tabularnewline
				\hline  
			\end{tabular}
		\par\end{centering}
	\end{table}
\par\end{center}

\begin{center}
	\begin{table}
		\caption{Monte Carlo results for the bootstrap estimator with regularly spaced times. The MCE estimated is $0.003$.}
		\label{IMA-MC-Table-RegularSpaced-b}
		\begin{centering}
			\begin{tabular}{|c|c|c|c|c|c|c|c|}
				\hline 
				N & $\theta_{0}$ & $\hat{\theta}^{\textrm{b}}$ & $\widehat{\textrm{se}}(\hat{\theta}^{\textrm{b}})$ & $\widetilde{\textrm{se}}(\hat{\theta}^{\textrm{b}})$ & $\textrm{bias}_{\hat{\theta}^{\textrm{b}}}$ & $\textrm{RMSE}_{\hat{\theta}^{\textrm{b}}}$ & $\textrm{CV}_{\hat{\theta}^{\textrm{b}}}$\tabularnewline
				\hline 
				\multirow{3}{*}{$100$} & $0.1$ & $0.120$ & $0.082$ & $0.073$ & $0.020$ & $0.084$ & $0.685$\tabularnewline
				\cline{2-8} \cline{3-8} \cline{4-8} \cline{5-8} \cline{6-8} \cline{7-8} \cline{8-8} 
				& $0.5$ & $0.502$ & $0.093$ & $0.099$ & $0.002$ & $0.093$ & $0.186$\tabularnewline
				\cline{2-8} \cline{3-8} \cline{4-8} \cline{5-8} \cline{6-8} \cline{7-8} \cline{8-8} 
				& $0.9$ & $0.915$ & $0.049$ & $0.050$ & $0.015$ & $0.051$ & $0.054$\tabularnewline
				\hline 
				\multirow{3}{*}{$500$} & $0.1$ & $0.098$ & $0.041$ & $0.042$ & $-0.002$ & $0.042$ & $0.422$\tabularnewline
				\cline{2-8} \cline{3-8} \cline{4-8} \cline{5-8} \cline{6-8} \cline{7-8} \cline{8-8} 
				& $0.5$ & $0.499$ & $0.039$ & $0.041$ & $-0.001$ & $0.039$ & $0.079$\tabularnewline
				\cline{2-8} \cline{3-8} \cline{4-8} \cline{5-8} \cline{6-8} \cline{7-8} \cline{8-8} 
				& $0.9$ & $0.902$ & $0.020$ & $0.021$ & $0.002$ & $0.021$ & $0.023$\tabularnewline
				\hline 
				\multirow{3}{*}{$1500$} & $0.1$ & $0.098$ & $0.026$ & $0.026$ & $-0.002$ & $0.026$ & $0.260$\tabularnewline
				\cline{2-8} \cline{3-8} \cline{4-8} \cline{5-8} \cline{6-8} \cline{7-8} \cline{8-8} 
				& $0.5$ & $0.500$ & $0.022$ & $0.023$ & $0.000$ & $0.022$ & $0.045$\tabularnewline
				\cline{2-8} \cline{3-8} \cline{4-8} \cline{5-8} \cline{6-8} \cline{7-8} \cline{8-8} 
				& $0.9$ & $0.902$ & $0.011$ & $0.012$ & $0.002$ & $0.011$ & $0.013$\tabularnewline
				\hline 
			\end{tabular}
		\par\end{centering}
	\end{table}
\par\end{center}

\section{Measures of performance in comparison with CARMA(2,1) model\label{sec:IMA-CARMA-App}}

\begin{center}
	\begin{table}
		\caption{Monte Carlo results of the experiment to compare the performance of the IMA and CARMA(2,1) models in fitting IMA model. The goodness of fit measure estimated from both models is the AIC.}
		\label{IMA-CARMA-App1}
		\begin{centering}
			\begin{tabular}{|c|c|c|c|c|c|c|c|c|}
				\hline 
				N & $\theta_{0}$ & $\hat{\theta}$ & $\widehat{\textrm{se}}(\hat{\theta})$ & AIC\_IMA & SD(AIC\_IMA) & AIC\_C21 & SD(AIC\_C21) &  \% Success  \tabularnewline
				\hline 
				\multirow{4}{*}{$300$} & 0.95  & 0.9437 & 0.0568 & 399.4029 & 14.1419 & 416.4938 & 15.3710 & 0.9930 \tabularnewline
				\cline{2-9} \cline{3-9} \cline{4-9} \cline{5-9} \cline{6-9} \cline{7-9} \cline{8-9} \cline{9-9} 
				& 0.9  & 0.8912 & 0.0534 & 407.1385 & 14.0296 & 419.1464 & 22.0460 & 0.9940 \tabularnewline
				\cline{2-9} \cline{3-9} \cline{4-9} \cline{5-9} \cline{6-9} \cline{7-9} \cline{8-9} \cline{9-9}  
				& 0.7 & 0.6762 & 0.1171 & 414.5499 & 13.9158 & 421.0512 & 19.3402 & 0.9750 \tabularnewline
								\cline{2-9} \cline{3-9} \cline{4-9} \cline{5-9} \cline{6-9} \cline{7-9} \cline{8-9} \cline{9-9} 
				& 0.5 & 0.4900 & 0.1346 & 416.5100 & 13.8214 & 421.8210 & 26.3998 & 0.9740 \tabularnewline
				\hline 
				\multirow{4}{*}{$1000$} & 0.95 & 0.9464 & 0.0473 & 1327.2653 & 43.8182 & 1383.1773 & 77.7452 & 0.9970 \tabularnewline
				\cline{2-9} \cline{3-9} \cline{4-9} \cline{5-9} \cline{6-9} \cline{7-9} \cline{8-9} \cline{9-9} 
				& 0.9  & 0.8971 & 0.0399 & 1352.7468 & 44.0957 & 1387.5857 & 79.1279 & 0.9970 \tabularnewline
				\cline{2-9} \cline{3-9} \cline{4-9} \cline{5-9} \cline{6-9} \cline{7-9} \cline{8-9} \cline{9-9} 
				 & 0.7 & 0.6924 & 0.0675 & 1379.2990 & 44.4712 & 1393.3974 & 54.6430 & 0.9910 \tabularnewline
				\cline{2-9} \cline{3-9} \cline{4-9} \cline{5-9} \cline{6-9} \cline{7-9} \cline{8-9} \cline{9-9} 
				& 0.5 & 0.4949 & 0.0737 & 1385.7829 & 44.4713 & 1395.4456 & 46.6157 & 0.9930 \tabularnewline
				\hline 
			\end{tabular}
		\par\end{centering}
	\end{table}
\par\end{center}

\begin{center}
	\begin{table}
		\caption{Monte Carlo results of the experiment to compare the performance of the IMA and CARMA(2,1) models in fitting IMA model. The goodness of fit measure estimated from both models is the -2 Log-likelihood.}
		\label{IMA-CARMA}
		\begin{centering}
			\begin{tabular}{|c|c|c|c|c|c|c|c|c|}
				\hline 
				N & $\theta_{0}$ & $\hat{\theta}$ & $\widehat{\textrm{se}}(\hat{\theta})$ & LL\_IMA & SD(LL\_IMA) & LL\_C21 & SD(LL\_C21) & \% Success \tabularnewline
				\hline 
				\multirow{4}{*}{$300$} & 0.95 & 0.9437 & 0.0568 & 397.4049 & 14.0854 & 411.4988 & 15.2356 & 0.9860 \tabularnewline
				\cline{2-9} \cline{3-9} \cline{4-9} \cline{5-9} \cline{6-9} \cline{7-9} \cline{8-9} \cline{9-9} 
				& 0.9 & 0.8912 & 0.0534 & 405.1405 & 13.9716 & 414.1514 & 21.9512 & 0.9470 \tabularnewline
				\cline{2-9} \cline{3-9} \cline{4-9} \cline{5-9} \cline{6-9} \cline{7-9} \cline{8-9} \cline{9-9}  
				& 0.7 & 0.6762 & 0.1171 & 412.5519 & 13.8562 & 416.0562 & 19.2315 & 0.8050 \tabularnewline
								\cline{2-9} \cline{3-9} \cline{4-9} \cline{5-9} \cline{6-9} \cline{7-9} \cline{8-9} \cline{9-9} 
				& 0.5 & 0.4900 & 0.1346 & 414.5120 & 13.7611 & 416.8260 & 26.3202 & 0.6730 \tabularnewline
				\hline 
				\multirow{4}{*}{$1000$} & 0.95 & 0.9464 & 0.0473 & 1325.2673 & 43.7576 & 1378.1823 & 77.6563 & 0.9970\tabularnewline
				\cline{2-9} \cline{3-9} \cline{4-9} \cline{5-9} \cline{6-9} \cline{7-9} \cline{8-9} \cline{9-9} 
				& 0.9 & 0.8971 & 0.0399 & 1350.7488 & 44.0343 & 1382.5907 & 79.0402 & 0.9970 \tabularnewline
				\cline{2-9} \cline{3-9} \cline{4-9} \cline{5-9} \cline{6-9} \cline{7-9} \cline{8-9} \cline{9-9} 
				 & 0.7 & 0.6924 & 0.0675 & 1377.3010 & 44.4091 & 1388.4024 & 54.5155 & 0.9600\tabularnewline
				\cline{2-9} \cline{3-9} \cline{4-9} \cline{5-9} \cline{6-9} \cline{7-9} \cline{8-9} \cline{9-9} 
				& 0.5 & 0.4949 & 0.0737 & 1383.7849 & 44.4089 & 1390.4506 & 46.4659 & 0.8920\tabularnewline
				\hline 
			\end{tabular}
		\par\end{centering}
	\end{table}
\par\end{center}

\begin{center}
	\begin{table}
		\caption{Monte Carlo results of the experiment to estimate the parameters of the IMA and CARMA(2,1) models in a simulated IMA process.}
		\label{IMA-CARMA-App2-2}
		\begin{centering}
			\begin{tabular}{|c|c|c|c|c|c|c|c|c|c|}
				\hline 
				N & $\theta_{0}$ & $\hat{\theta}$ & $SD(\hat{\theta})$  & $\hat{\phi_1^C}$ & $SD(\hat{\phi_1^C})$& $\hat{\phi_2^C}$ & $SD(\hat{\phi_2^C})$ & $\hat{\theta_1^C}$ & $SD(\hat{\theta_1^C})$  \tabularnewline
				\hline 
				\multirow{4}{*}{$300$}& 0.95 & 0.9437 & 0.0568 & 1.9772 & 2.4118 & 4.4998 & 2.0976 & 2.4048 & 5.2211  \tabularnewline
				\cline{2-9} \cline{3-9} \cline{4-9} \cline{5-9} \cline{6-9} \cline{7-9} \cline{8-9} \cline{9-9} \cline{10-10} 
				& 0.9 & 0.8912 & 0.0534 & 1.9880 & 2.6548 & 4.3337 & 2.2127 & 2.1737 & 4.4533 \tabularnewline
				\cline{2-9} \cline{3-9} \cline{4-9} \cline{5-9} \cline{6-9} \cline{7-9} \cline{8-9} \cline{9-9}  \cline{10-10} 
				& 0.7 & 0.6762 & 0.1171 & 1.1625 & 3.6967 & 3.2940 & 2.7077 & 5.5312 & 13.1653 \tabularnewline
								\cline{2-9} \cline{3-9} \cline{4-9} \cline{5-9} \cline{6-9} \cline{7-9} \cline{8-9} \cline{9-9} \cline{10-10} 
				& 0.5 & 0.4900 & 0.1346 & 1.6704 & 3.5028 & 2.8492 & 2.5655 & 4.1638 & 9.0002 \tabularnewline
				\hline 
				\multirow{4}{*}{$1000$} & 0.95 & 0.9464 & 0.0473 & 2.9784 & 2.2235 & 5.5765 & 1.9523 & 2.1122 & 5.9226 \tabularnewline
				\cline{2-9} \cline{3-9} \cline{4-9} \cline{5-9} \cline{6-9} \cline{7-9} \cline{8-9} \cline{9-9} \cline{10-10} 
				& 0.9  & 0.8971 & 0.0399 & 3.2768 & 2.3768 & 5.5184 & 2.1549 & 1.4643 & 3.4077  \tabularnewline
				\cline{2-9} \cline{3-9} \cline{4-9} \cline{5-9} \cline{6-9} \cline{7-9} \cline{8-9} \cline{9-9} \cline{10-10} 
				 & 0.7 & 0.6924 & 0.0675 & 3.7004 & 3.2235 & 5.1538 & 2.4396 & 1.1568 & 3.4549 \tabularnewline
				\cline{2-9} \cline{3-9} \cline{4-9} \cline{5-9} \cline{6-9} \cline{7-9} \cline{8-9} \cline{9-9} \cline{10-10} 
				& 0.5 & 0.4949 & 0.0737 & 3.5409 & 3.5898 & 4.4348 & 2.6495 & 1.4821 & 3.6772  \tabularnewline
				\hline 
			\end{tabular}
		\par\end{centering}
	\end{table}
\par\end{center}

\newpage
\bibliography{mybibfile}

\end{document}